\documentclass[11pt, oneside]{article} 

\usepackage[utf8]{inputenc}
\usepackage[T1]{fontenc}
\usepackage{xcolor} 
\usepackage{color} 

\usepackage{amsmath}
\usepackage{amssymb} 
\usepackage{mathrsfs} 
\usepackage{bm} 
\usepackage{bbm} 

\usepackage{amsthm} 
\theoremstyle{plain}
\newtheorem{theorem}{Theorem} 
\newtheorem{corollary}[theorem]{Corollary} 

\newtheorem{lemma}[theorem]{Lemma}

\newtheorem{remark}[theorem]{Remark}

\usepackage{algorithm2e} 
\SetKwInput{KwInputs}{Inputs}
\usepackage[noend]{algpseudocode}
\makeatletter
\def\BState{\State\hskip-\ALG@thistlm}
\makeatother

\usepackage{graphicx} 
\usepackage{subcaption} 

\usepackage{tikz} 
\usetikzlibrary{positioning ,shadows,matrix}

\usepackage[english]{babel} 

\usepackage[numbers]{natbib} 
\usepackage[nottoc]{tocbibind} 

\usepackage{hyperref} 
\hypersetup{
    pdfpagemode={UseOutlines},
    bookmarksopen,
    pdfstartview={FitH},
    colorlinks,
    linkcolor={blue}, 
    citecolor={blue}, 
    urlcolor={blue} 
}

\usepackage[letterpaper,margin=1.25in]{geometry} 

\usepackage{enumitem} 

\def\ind{\overset{\textnormal{ind}}{\sim}} 
\def\iid{\overset{\textnormal{iid}}{\sim}} 
\def\eps{\varepsilon} 

\makeatletter  
\@ifundefined{@currsizeindex} 
  {\RequirePackage{relsize}\let\dolarger\relsize} 
  {\def\dolarger#1{\larger[#1]}} 
\newcommand*\@@bigtimes[2]{\vphantom{\prod} 
  \vcenter{\hbox{\dolarger{4}$\m@th#1\mkern-2mu\times\mkern-2mu$}}} 
\newcommand*\bigtimes{\mathop{\mathpalette\@@bigtimes\relax}\displaylimits} 
\makeatother

\def\N{\mathbb{N}}\def\R{\mathbb{R}}\def\1{\mathbbm{1}}

\def\Ccal{\mathcal{C}}\def\Fcal{\mathcal{F}}\def\Pcal{\mathcal{P}}\def\Wcal{\mathcal{W}}

\def\erfc{\textnormal{erfc}}\def\sign{\textnormal{sign}}

\def\beq{\begin{equation}}
\def\eeq{\end{equation}}
\def\beqs{\begin{equation*}}
\def\eeqs{\end{equation*}}
\def\bal{\begin{align}}
\def\bals{\begin{align*}}

\title{\bf On strong posterior contraction rates for Besov-Laplace priors in the white noise model}
\author{Emanuele Dolera$^*$, Stefano Favaro$^{\star,\dagger}$ and Matteo Giordano$^\star$ \\ \\ 
$^*$University of Pavia, $^{\star}$University of Turin and $^{\dagger}$Collegio Carlo Alberto, Turin}
\date{} 

\begin{document}

\maketitle

\abstract{
In this article, we investigate the problem of estimating a spatially inhomogeneous function and its derivatives in the white noise model using Besov-Laplace priors. We show that smoothness-matching priors attains minimax optimal posterior contraction rates, in strong Sobolev metrics, over the Besov spaces $B^\beta_{11}$, $\beta > d/2$, closing a gap in the existing literature. Our strong posterior contraction rates also imply that the posterior distributions arising from Besov-Laplace priors with matching regularity enjoy a desirable plug-in property for derivative estimation, entailing that the push-forward measures under differential operators optimally recover the derivatives of the unknown regression function. The proof of our results relies on the novel approach to posterior contraction rates, based on Wasserstein distance, recently developed by Dolera, Favaro and Mainini (\textit{Probability Theory and Related Fields}, 2024). We show how this approach allows to overcome some technical challenges that emerge in the frequentist analysis of smoothness-matching Besov-Laplace priors.
}

\bigskip


\bigskip


%


\section{Introduction}

Besov priors are a popular family of prior probability distributions for functions. They were first introduced by Lassas et al.~\cite{LSS09}, and have since then gained a widespread use in Bayesian inverse problems and medical imaging applications; see e.g.~\cite{LP01,BD06,VLS09,DHS12,R13,SE15,JPG16,DS17,ABDH18,KLSS23} as well as the overviews \cite{S10,DS17}. Besov priors are defined through random expansions on wavelet bases, assigning to the wavelet coefficients independent priors of $p$-exponential type, i.e.~with tails between the Gaussian (corresponding to $p=2$) and the Laplace distribution (corresponding to $p=1$). Among these, the Laplace case, which yields `Besov-Laplace priors', has by far enjoyed the greater success, furnishing a discretisation-invariant version \cite{LS04,LSS09} of the celebrated (finite-dimensional) total-variation prior of Rudin et al.~\cite{ROF92}, while also preserving a logarithmically concave structure that enhances analytical tractability \cite{DHS12,ADH21} and posterior computations \cite{BG15,chen2018dimension}.

	In the aforementioned application areas, Besov-Laplace priors exhibit excellent empirical performances in the recovery of `spatially inhomogeneous' objects of `bounded variation type' \cite{RVJKLMS06,LSS09,DHS12}, that is ones that are generally flat or smooth across the domain but may have localised sharp variations or even discontinuities, such as images, layered media or spiky signals \cite[Section 1]{DJ98}. Such scenarios are known to be well modelled by functions belonging to Besov spaces $B^\beta_{1q}$, $\beta\ge0$, $q\in[1,\infty]$, which measure the spatial variability in an $L^1$-sense and quantify the smoothness level through the $\ell^1$-decay of the wavelet coefficients (cf.~Section \ref{Sec:Notation} for details). In particular, the space of bounded variation functions $BV$, which is of specific interest for imaging and signal processing, is known to be related to the Besov scale through the (strict) inclusions $B^1_{11}\subset BV\subset B^1_{1\infty}$ (e.g.~\cite[Section 1]{DJ98}). Since Besov-Laplace priors induce $\ell^1$-type penalties on the wavelet coefficients, they furnish a natural Bayesian model for functions in the spaces $B^\beta_{1q}$, achieving attractive sparsity-promoting and edge-preserving properties, particularly at the level of the maximum-a-posteriori, MAP, estimators \cite{ABDH18}. On the contrary, Gaussian priors induce $\ell^2$-type penalties and thus are suited to model Sobolev-regular functions with smaller variations, and have been shown, both empirically and theoretically, to perform sub-optimally in structured recovery problems \cite{ADH21,giordano2022inability,abraham2023deep,agapiou2024laplace}. 
	
%
%
%

\subsection{Existing literature and motivation}
\label{Sec:Literature}

In the article, we are interested in the performance of posterior-based inference with Besov-Laplace priors for the recovery of (possibly) spatially inhomogeneous functions, in the large sample size limit and under the frequentist assumption that the data have been generated by some fixed `ground truth' belonging to a Besov space. Such investigation has been recently initiated by Agapiou et al.~\cite{ADH21}, building on the seminal advances in the theory of frequentist analyses of Bayesian nonparametric procedures achieved during the last two decades \cite{GGvdV00,shen2001rates,GvdV07,vdVvZ08,GN11}; see also the monograph \cite{GvdV17}, as well as \cite[Chapter 7.3]{GN16}. For Besov priors and spatially inhomogeneous ground truths, \cite{ADH21} proved minimax-optimal posterior contraction rates in the white noise model, later extended by \cite{GR22,giordano23besov,agapiou2024laplace} to, respectively, non-linear inverse problems, density estimation, and drift estimation for multi-dimensional diffusions. Further, adaptive posterior contraction rates with hierarchical procedures based on Besov priors were obtained in \cite{giordano23besov,agapiou2024adaptive}. We also mention the earlier related results by \cite{CN14,R13,AGR13}.

	A common feature in this recent line of work is the required tuning of the hyperparameters of the employed Besov priors, as across the statistical models considered in \cite{ADH21,GR22,giordano23besov,agapiou2024laplace} optimal posterior contraction rates are obtained exclusively for `undersmoothing and rescaled' priors. These are priors whose draws possess strictly lower regularity than the inferential target, and are further rescaled by a diverging power of the sample size, causing them to asymptotically vanish almost surely \citep[Section 3.1.2]{giordano23besov}. Such a prior construction is somewhat artificial, and it is in contrast with the broader literature on Gaussian priors, where optimal posterior contraction rates are typically attained with smoothness-matching priors \cite{vdVvZ08} (albeit, under traditional Sobolev- or H\"older-type regularity assumptions on the ground truth). Also, the combination of undersmoothing and re-scaling has the undesirable consequence of creating serious theoretical and computational challenges in the pursuit of adaptation, possibly necessitating the use of hierarchical models with two hyperparameters \cite{agapiou2024adaptive} or $n$-dependent hyperpriors \cite{giordano23besov}. On the other hand, the only available results for smoothness-matching Besov priors are derived, in the white noise model, in \cite[Section 5.1]{ADH21}, where `polynomially' sub-optimal posterior contraction rates are obtained. The origin of this sub-optimality is discussed in \cite[p.~21]{ADH21} (see also the remarks after Theorem 3.1 in \cite{giordano23besov}) and, in essence, lies in the involved structure of the `sieves sets' associated to Besov priors, which complicates the verification of the testing condition of \cite{GGvdV00} for posterior contraction via metric entropy bounds. Further details can be found in the discussion after Theorem \ref{Theo:Main} below. The derivation of optimal posterior contraction rates for Besov priors with matching regularity thus represents a notable gap in the existing literature \citep[Remark 2]{giordano23besov}. We will provide an answer in the positive to this open question.

%
%
%

\subsection{Our contributions}

We will re-examine the frequentist asymptotic performance of smoothness-matching Besov-Laplace priors in the white noise model, investigating the rate of posterior contraction towards (possibly) spatially inhomogeneous ground truths. Further, we will consider the problem of estimating the derivatives of the unknown regression function, for which we will adopt the simple (Bayesian) `plug-in' strategy consisting in differentiating the draws from the posterior distribution \cite{stein2012interpolation,holsclaw2013gaussian,liu2022optimal}, which is conceptually attractive since it does not require the construction (and implementation) of a separate statistical procedure. Our interest in the problem of derivative estimation is largely motivated by the natural application areas of Besov priors, where gradients often carries fundamental information that should be factor into the analysis (for example, as in the widely used total variation regularisation in imaging, \cite{ROF92}). Derivatives also emerge as the key quantities regulating the rate of change of unknown surfaces in a number of scientific disciplines, including geology, meteorology, and the environmental studies; see \cite{liu2022optimal} and references therein for a comprehensive discussion on derivative estimation.

	In our main result, Theorem \ref{Theo:Main}, we will show that smoothness-matching Besov-Laplace priors attain minimax-optimal posterior contraction rates towards ground truths in the Besov scale, without the need of the specific undersmoothing and rescaling employed in \cite{ADH21,GR22,giordano23besov,agapiou2024laplace}. In fact, for true regression functions $f_0\in B^\beta_{11}$, $\beta>d/2$, we will, unlike the previous references, obtain `strong' posterior contraction rates with respect to Sobolev $H^s$-norms, simultaneously over the natural range $s\in[0,\beta-d/2)$. This result will then imply that the plug-in posterior distributions optimally solve the derivatives estimation problem (cf.~Corollary \ref{Cor:DerivEstim}). With respect to the problem of derivative estimation, Liu and Li \cite{liu2022optimal} have recently shown that such `plug-in' property is enjoyed by Gaussian process regression (under Sobolev-type regularity assumptions), and we present here, to our knowledge, the first extension to Besov-Laplace priors in the spatially inhomogeneous setting.

	For the proof of Theorem \ref{Theo:Main}, we will adopt the novel approach to posterior contraction rates, based on Wasserstein distance, recently developed by Dolera et al.~\cite{dolera2024strong}. Our motivation is twofold: firstly, by construction, the strategy of \cite{dolera2024strong} is inherently flexible in the choice of the metric in which the speed of posterior contraction is measured, allowing in the present setting to  simultaneously handle a large collection of $H^s$-metrics. 
Secondly, as discussed in Section \ref{Sec:Literature}, in the presence of spatially inhomogeneous ground truths, the pursuit of the testing-based strategy of \cite{GGvdV00} for smoothness-matching Besov priors presents some technical challenges that generally lead to sub-optimal rates. Through the approach of \cite{dolera2024strong} we will be able, in the problem at hand, to circumvent these difficulties by employing a set of different tools that include Lipschitz properties of posterior distributions with respect to the Wasserstein distance \cite{dolera2023lipschitz}, estimates for the Poincaré constants of log-concave probability measures \cite{cattiaux2022functional}, and the so-called Laplace methods for integral approximation \cite{breitung2006asymptotic}. Thus, an additional contribution of the present article is to demonstrate how the approach of \cite{dolera2024strong} may represent a useful alternative to the well-established testing theory for posterior contraction rates, in settings where the verification of the conditions of the latter may prove to be out of reach. For the sake of presentation and the convenience of the reader, we have dedicated Section \ref{Sec:WassDyn} to providing an overview of the general argument developed in \cite{dolera2024strong}, highlighting the structure that emerges in the problem of deriving strong posterior contraction rates for Besov-Laplace priors in the white noise model.

%
%
%

\subsection{Organisation of the paper}

In Section \ref{Sec:WhiteNoise}, we formally introduce the statistical model in which we will formulate our results, alongside some necessary preliminary definitions and notation. We then provide, in Section \ref{Sec:WassDyn}, an overview of the Wasserstein dynamics approach to posterior contraction rates. In Section \ref{Sec:Main}, we introduce Besov-Laplace priors, and then state our main results concerning the asymptotic behaviour of the associated posterior distributions. We conclude in Section \ref{Sec:Discussion} with a summary and an overview on some related open research directions. Section \ref{Sec:Proofs} contains the proofs of the main technical lemmas for the proof of Theorem \ref{Theo:Main}.

%
%
%
%
%


\section{On the Wasserstein dynamics approach in the white noise model}
\label{Sec:WhiteNoise}

%
%
%

\subsection{Preliminaries and notation}
\label{Sec:Notation}

Throughout, we will consider function spaces defined on the $d$-dimensional unit cube, $[0,1]^d$, $d\in\N$. Denote by $L^p\equiv L^p([0,1]^d)$, $p\in[1,\infty]$, the usual Lebesgue spaces, equipped with norm $\|\cdot\|_{L^p}$; further denote by $\langle \cdot,\cdot\rangle_{L^2}$ the $L^2$-inner product.

	Let $\{\psi_l\}_{l=1}^\infty$ be (a singe-index reordering of) a tensor-product orthonormal wavelet basis of $L^2$, constructed from $S$-regular ($S\in\N$) compactly supported and boundary corrected Daubechies wavelets in $L^2([0,1])$; see e.g.~\cite[Appendix A]{LSS09} and \cite[Chapter 4.3]{GN16} for details. For real numbers $s\in[0,S]$ and $p\in[1,\infty)$ (respectively, the `smoothness' and `integrability' parameters), define the Besov spaces $B^s_p\equiv B^s_{pp}([0,1]^d)\subset L^p$ by
\beq
\label{Eq:BesovSpaces}
	B^s_p 
	:= \left\{f=\sum_{l=1}^\infty f_l \psi_l : \| f \|^p_{B^s_p}
	:=\sum_{l=1}^\infty l^{p(s/d+1/2)-1}|f_l|^p<\infty \right\}.
\eeq
Since $S$ can be taken to be arbitrarily large when constructing the underlying wavelet basis, we will tacitly assume the condition $s\le S$ to be satisfied throughout this article.

	For $p=2$, \eqref{Eq:BesovSpaces} corresponds to the wavelet characterisation of the traditional Hilbert-Sobolev spaces $H^s\equiv H^s([0,1]^d)$ and their norms $\|\cdot\|_{H^s}$. For $p<2$, the spaces $B^s_p$ contain functions that may be `spatially inhomogeneous', i.e.~flat in some parts of the domain and irregular (or discontinuous) in others, cf.~\cite[Section 1]{DJ98}. Among the latter, the most important instances are for $p=1$; in particular, $B^1_1$ is closely related to the space of bounded variation function, e.g.~\cite[Section 1]{DJ98}. Note that \eqref{Eq:BesovSpaces} defines a special case of the more general Besov spaces $B^s_{pq}\equiv B^s_{pq}([0,1]^d)$, $p,q\in[1,\infty]$, cf.~\cite[Chapter 4.3]{GN16}. Here, we restrict to $p=q$ for conciseness; note that the minimax rates of estimation over balls in $B^s_{pq}$ are generally known to be driven only by the first index $p$, e.g.~\cite{DJ98}.

	Denote by $\ell^p$, $p\in[1,\infty]$, the space of $p$-summable sequences of real numbers. The map $f\mapsto (f_l)_{l=1}^\infty$ establishes an isometry between $B^s_p$ and the corresponding Besov sequence space $b^s_p\subset \ell^p$. We will often identify $f$ with its sequence of wavelet coefficients, writing $f=(f_l)_{l=1}^\infty$, and the function spaces $B^s_p$ with the sequence spaces $b^s_p$.

	Let us recall the definition of the Wasserstein distances. For a separable metric space $(\Theta,\rho)$, denote by $\Pcal(\Theta)$ the collection of all Borel probability measures on $\Theta$; then, for any $q\in[1,\infty)$, the $q$-Wasserstein distance $\Wcal_q(\mu_1,\mu_2)$ between two probability measures $\mu_1,\mu_2$ on $\Theta$ satisfying
\[
	\int_\Theta \rho(\theta,\theta^*_i)^q\mu_i(d\theta)<\infty,
	\qquad \text{for some $\theta^*_i\in\Theta$}, 
	\qquad i=1,2,
\]
is defined as
\begin{equation}
\label{Eq:WassDist}
	\Wcal_q(\mu_1,\mu_2):=\inf_{\nu\in\Fcal(\mu_1,\mu_2)}
	\left(\int_{\Theta^2} \rho(x,y)^q\nu(dxdy) \right)^\frac{1}{q},
\end{equation}
where $\Fcal(\mu_1,\mu_2)$ is the family of all probability measures on $\Theta^2$ with $i^{\text{th}}$-marginal equal to $\mu_i$, $i=1,2$. See \cite[Chapter 7]{ambrosio2008gradient}.

	For absolutely continuous probability distributions $\pi$ on $\R$ we will use the same notation for the associated probability density functions $\pi(x), x\in\R$. We will denote by $\lesssim,\gtrsim$ and $\simeq$, respectively, one- or two-sided inequalities holding up to multiplicative constants. We will write $c_1,c_2,\dots$ for positive numerical constants.
	
%
%
%

\subsection{The Gaussian white noise model, product priors and posteriors}
\label{Sec:ModelPriorPost}

In the present article, our main focus is on the (Gaussian) white noise model, which consists in problem of estimating an unknown function $f\in L^2$ from observations
\beq
\label{Eq:WhiteNoise}
	X^{(n)}:=f + \frac{1}{\sqrt n}W, 
	\qquad n\in\N,
\eeq
where $W$ is a Gaussian white noise process indexed by $L^2$. Such observation scheme is canonically used in statistical theory for evaluating the performance of nonparametric procedures, e.g.~\cite{T09,GN16,GvdV17}. It is known to be asymptotically equivalent to the standard nonparametric regression model with Gaussian measurement errors under suitable assumptions on the design \cite{brown2002asymptotic,R08}; thus, up to correctly handling discretisation, similar results to the ones to follow may be expected to also hold for the latter problem.

	Given an orthonormal basis $\{\psi_l\}_{l=1}^\infty$ of $L^2$ (below, taken to be a wavelet basis as described in Section \ref{Sec:Notation}), the white noise model \eqref{Eq:WhiteNoise} may equivalently be written as a (Gaussian) sequence model by setting
\[
	X^{(n)}_l:=\langle \psi_l,X^{(n)}\rangle_{L^2}
	=\langle f,\psi_l\rangle_{L^2} + \frac{1}{\sqrt n}\langle \psi_l,W\rangle_{L^2},
	\qquad l\in\N,
\]
whereupon, denoting by $f_l:=\langle f,\psi_l\rangle_{L^2}$ and $W_l:=\langle \psi_l,W\rangle_{L^2}$, the random sequence $(X^{(n)}_l)_{l=1}^\infty$ is seen to satisfy
\beq
\label{Eq:GaussSeq}
	X^{(n)}_l = f_l + \frac{1}{\sqrt n}W_l, 
	\qquad l,n\in\N, 
	\qquad (f_l)_{l=1}^\infty\in\ell^2,
	\qquad W_l\iid N(0,1).
\eeq
In this context, observing $X^{(n)}$ in the white noise model \eqref{Eq:WhiteNoise} may be understood as observing a realisation of the random sequence $(X^{(n)}_l)_{l=1}^\infty$. We will accordingly identify $X^{(n)}=(X^{(n)}_l)_{l=1}^\infty$, denoting by $P^{(n)}_f$ the joint (Gaussian product) law of $X^{(n)}$, and by $E^{(n)}_f$ the expectation with respect to it.

	A popular approach to prior modelling in the white noise model \eqref{Eq:WhiteNoise} is via random basis expansions with independent coefficients, e.g.~\cite[Chapter 2.1]{GvdV17}, letting the prior $\Pi$ arise as the law of the random function
\beq
\label{Eq:ProdPrior}
	f = \sum_{l=1}^\infty  f_l \psi_l, \qquad f_l \ind \pi_l,
\eeq
where $\{\psi_l\}_{l=1}^\infty$ is a set of `basis' functions in $L^2$ (again, below a wavelet basis), and where the independent marginal priors $\{\pi_l\}_{l=1}^\infty$ on $\R$ are specified so that $f$ in \eqref{Eq:ProdPrior} is in $L^2$ almost surely (and $\Pi$ is supported on $L^2$). Note that this corresponds to assigning to the sequence of coefficients $(f_l)_{l=1}^\infty$ the product prior $\pi:=\otimes_{l= 1}^\infty\pi_l$ on $\R^\infty$. If, in fact, $\{\psi_l\}_{l=1}^\infty$ is an orthonormal basis of $L^2$, then given observations $X^{(n)}$ from the white noise model \eqref{Eq:WhiteNoise} (equivalently, from \eqref{Eq:GaussSeq}), the posterior distribution $\Pi(\cdot|X^{(n)})$ of $f|X^{(n)}$ is obtained by Bayes formula, e.g.~\cite[Section 7.3]{GN16}, and in view of the underlying product structure may be seen to also arise as the law of a random basis expansion as in \eqref{Eq:ProdPrior}, now with independent marginal distributions
\beq
\label{Eq:MargPost}
	f_l|X^{(n)} \ind \pi_{n,l}(\cdot|X^{(n)}_l),
	\qquad 
	\pi_{n,l}(\theta|x)
	=
	\frac{e^{-\frac{n}{2}(\theta - x)^2}\pi_l(\theta)}
	{\int_\R e^{-\frac{n}{2}(\vartheta - x)^2}\pi_l(\vartheta)d\vartheta },
	\qquad \theta,x\in\R.
\eeq

%
%
%

\subsection{The Wasserstein dynamics approach to strong posterior contraction rates}
\label{Sec:WassDyn}

The overarching question investigated in the present paper is to characterise, for Besov-Laplace priors (introduced below in Section \ref{Sec:LaplPriors}), the asymptotic properties of the posterior under the frequentist assumption that the observations $X^{(n)}$ from the white noise model \eqref{Eq:WhiteNoise} are, in reality, generated by some fixed ground truth $f_0\in L^2$, i.e.~$X^{(n)}\sim P^{(n)}_{f_0}$, studying the rate of concentration of $\Pi(\cdot|X^{(n)})$ around $f_0$ as $n\to\infty$. We are interested in such investigation under the perspective of the recent line of work on nonparametric Bayesian procedures for the recovery of spatially inhomogeneous ground truths \cite{ADH21,GR22,giordano23besov,agapiou2024laplace,agapiou2024adaptive,agapiou2024heavy}, that is functions that may be flat in some areas and `spiky' in others. Similarly to the latter references, we will then assume that $f_0\in B^\beta_1$, for some $\beta\ge0$; see Section \ref{Sec:Notation} for definitions and details.

	Furthermore, we are, unlike the aforementioned contributions, interested in characterising the asymptotic concentration of the posterior with respect to `strong' Sobolev norms $\|\cdot\|_{H^s}, \ s\ge 0$ (which, for $s>0$, are indeed stronger than the typically employed $L^2$-metric, corresponding to $s=0$). This is motivated by the largely unexplored issue of nonparametric Bayesian derivative estimation, with Besov-Laplace priors, in the presence of spatial inhomogeneity. For example, in view of the equivalence for the first order Sobolev norm
\[
	\|f\|_{H^1}\simeq \|f\|_{L^2} + \|\nabla f\|_{L^2([0,1]^d;\R^d)}, \qquad f\in H^1,
\]
cf.~\cite[Chapter 4.3]{GN16}, the rate of posterior contraction in $H^1$-metric furnishes an upper bound for the speed of concentration of the so-called (Bayesian) `plug-in' strategy, which entails inferring the unknown gradient $\nabla f$ via the conditional distribution of $\nabla f|X^{(n)}$, obtained as the push forward measure of the posterior of $f|X^{(n)}$ under the gradient operator; see Section \ref{Sec:DerivEstim}. Recalling the compact embedding $B^\beta_1\subset H^s$ holding for all $\beta>s+d/2$, e.g.~\cite[p.33]{LSS09}, we then assume that $\beta>d/2$, and for $s\in[0, \beta-d/2)$ seek to derive $H^s$-posterior contraction rates, that is positive real sequences $\xi_n\to0$ such that, as $n\to\infty$
\beq
\label{Eq:PCR}
	E_{f_0}^{(n)}\Pi(f : \|f - f_0\|_{H^s}>\xi_n|X^{(n)})\to 0.
\eeq

	To do so, we will adopt the novel `Wasserstein dynamics' approach to posterior contraction rates in general metric spaces put forth in a recent article by Dolera et al.~\cite{dolera2024strong}. The rest of this section is devoted to providing an overview of the general argument developed in the latter reference, specialising it to the problem at hand, where some additional structure emerges and can be taken advantage of in the analysis.

	Assume that $X^{(n)}\sim P^{(n)}_{f_0}$ for some unknown $f_0\in B^\beta_1$, for some $\beta>d/2$. Let $\Pi$ be a Borel probability measure on $L^2$ with support contained in $H^s$, and let $\Pi(\cdot|X^{(n)})$ be the resulting posterior distribution. For $s\in [0,\beta-d/2)$, endow $B^\beta_1$ with the $\|\cdot\|_{H^s}$ norm and set
\beq
\label{Eq:Eps}
	\eps_n 
	:= E_{f_0}^{(n)}\left[\Wcal_2(\Pi(\cdot|X^{(n)}),\delta_{f_0}) \right],
\eeq
where the 2-Wasserstein distance $\Wcal_2(\cdot,\cdot)$ is defined as in \eqref{Eq:WassDist} with $\Theta = B^\beta_1$ and $\rho$ equal to the $H^s$-metric. Then, Lemma 2.3 in \cite{dolera2024strong} shows by simple arguments that the sequence $(\varepsilon_n)_{n=1}^\infty$ defined in \eqref{Eq:Eps} provides an upper bound for the $H^s$-posterior contraction rate, in that \eqref{Eq:PCR} can be proved to hold with $\xi_n = M_n\varepsilon_n$ for any positive real sequence $M_n\to\infty$. This transforms the issue into the derivation of suitable upper bounds for the expected Wasserstein distance in the right hand side of \eqref{Eq:Eps}. Dolera et al.~\cite{dolera2024strong} outlined how this can be generally approached through a clever use of sufficient statistics, local Lipschitz-continuity properties of posterior distributions, as well as powerful analytical and probabilistic tools that include concentration inequalities and large deviation principles, and the so-called Laplace method for the approximation of integrals. See \cite[Section 2]{dolera2024strong} for an overview. In the problem at hand, following ideas from \cite{dolera2024strong} we obtain by Jensen's inequality and the triangle inequality for the Wasserstein distance that
\beq
\label{Eq:TwoWass}
	\eps_n^2 
	\le E_{f_0}^{(n)}\left[\Wcal_2^2(\Pi(\cdot|X^{(n)}),\delta_{f_0}) \right]
	\le 2E_{f_0}^{(n)}\left[\Wcal_2^2(\Pi(\cdot|X^{(n)}),\Pi(\cdot|f_0)) \right]
	+2\Wcal_2^2(\Pi(\cdot|f_0),\delta_{f_0}),
\eeq
where $\Pi(\cdot|f_0)$ is the (deterministic) product measure on $L^2$ assigning to the wavelet coefficients $(f_l)_{l=1}^\infty$ the independent marginal distributions $f_l\ind \pi_{n,l}(\cdot|f_{0,l}) $ defined as in \eqref{Eq:MargPost} with $x=f_{0,l}$. To handle the first term in the right hand side of \eqref{Eq:TwoWass}, we note that, under $P^{(n)}_{f_0}$, $X^{(n)}\to f_0$ (component-wise) almost surely as $n\to\infty$, cf.~\eqref{Eq:GaussSeq}, and therefore we may expect the probability measures $\Pi(\cdot|X^{(n)})$ and $\Pi(\cdot|f_0)$ to also be close.  The second term  $\Wcal_2(\Pi(\cdot|f_0),\delta_{f_0})$ instead measures the speed at which $\Pi(\cdot|f_0)$ converges towards $\delta_{f_0}$, resulting in a purely analytical problem that requires the study of the asymptotics of the ratio between two vanishing integrals.

	In the present setting, this line of though can be further developed exploiting the specific structural interplay between the white noise model and product priors.  Let $\{\psi_l\}_{l=1}^\infty$ be an orthonormal wavelet basis of $L^2$ as introduced in Section \ref{Sec:Notation}, and let $\Pi$ be a product prior constructed according to \eqref{Eq:ProdPrior}. Then, since $\Pi(\cdot|X^{(n)})=\otimes_{l=1}^\infty \pi_{n,l}(\cdot|X^{(n)}_l)$ and $\Pi(\cdot|f_0)=\otimes_{l=1}^\infty \pi_{n,l}(\cdot|f_{0,l})$ are both in product forms, it holds that 
	\beq
\label{Eq:FirstBound}
	\Wcal_2^2(\Pi(\cdot|X^{(n)}),\Pi(\cdot|f_0))
	=\sum_{l=1}^\infty l^{2s/d}
	\Wcal_2^2(\pi_{n,l}(\cdot|X^{(n)}_l),\pi_{n,l}(\cdot|f_{0,l})),
\eeq
see e.g.~\cite[Section 2]{panaretos2019statistical}, where in the last line the 2-Wasserstein distances are to be intended (in slight abuse of notation) as being between Borel probability measures on $\R$. Further, since
\[
	\Wcal_2^2(\Pi(\cdot|f_0),\delta_{f_0})
	= \int_{B^s_p}\|f - f_0\|^2_{H^s}\Pi(df|X^{(n)}),
\]
cf.~the proof of Lemma 2.3 in \cite{dolera2024strong}, we have, in view of the wavelet characterisation of the $H^s$-norm and the product structure of $\Pi(\cdot|f_0)$,
\bals
	\Wcal_2^2(\Pi(\cdot|f_0),\delta_{f_0})
	&=
	\int_{B^s_p}\sum_{l=1}^\infty l^{2s/d}|f_l - f_{0,l}|^2\Pi(df|f_0)\\
	&=\sum_{l=1}^\infty l^{2s/d} \int_\R |\theta - f_{0,l}|^2\pi_{n,l}(\theta|f_0)
	=\sum_{l=1}^\infty l^{2s/d}\Wcal_2^2(\pi_{n,l}(\cdot|f_{0,l}),\delta_{f_{0,l}})
	.
\end{align*}
Combining the latter identity with \eqref{Eq:TwoWass} and \eqref{Eq:FirstBound} then yields
\beq
\label{Eq:TwoSeries}
\begin{split}
	\eps_n^2
	\le 
	2\sum_{l=1}^\infty l^{2s/d}E_{f_0}^{(n)}
	\left[ \Wcal_2^2(\pi_{n,l}(\cdot|X^{(n)}_l),\pi_{n,l}(\cdot|f_{0,l})\right]
	+2\sum_{l=1}^\infty l^{2s/d}\Wcal_2^2(\pi_{n,l}(\cdot|f_{0,l}),\delta_{f_{0,l}}).
\end{split}
\eeq
This shows that, in the white noise model, the problem of deriving (upper bounds to) $H^s$-posterior contraction rates for general product priors can be reduced to the study of the asymptotics of two specific sequences of Wasserstein distances between univariate distributions, thereby reducing, in essence, the issue from an infinite-dimensional setting to a finite-dimensional one, which significantly enhance analytical tractability.

	Below, we will upper bound the terms $\Wcal_2^2(\pi_{n,l}(\cdot|X^{(n)}_l),\pi_{n,l}(\cdot|f_{0,l}))$ appearing in the first series of \eqref{Eq:TwoSeries} by deriving Lipschitz properties of the marginal posteriors with respect to the 2-Wasserstein distance, that is
\beq\label{Eq:LipProp}
	\Wcal_2(\pi_{n,l}(\cdot|X^{(n)}_l),\pi_{n,l}(\cdot|f_{0,l}))
	\le C_{n,l}|X^{(n)}_l - f_{0,l}|,
\eeq
for constants $C_{n,l}>0$ that can be explicitly characterised using some recent results in the literature on functional inequalities and Poincaré constant estimates for log-concave measures \cite{cattiaux2022functional,dolera2023lipschitz}. See the proof of Lemma \ref{Lem:StochPart}. To upper bound the second series of \eqref{Eq:TwoSeries}, we will study the precise asymptotics (as $l,n\to\infty$) of the terms
\[
	\Wcal_2^2(\pi_{n,l}(\cdot|f_{0,l}),\delta_{f_{0,l}})
	=
	\frac{\int_\R |\theta-f_{0,l}|^2e^{-\frac{n}{2}(\theta - f_{0,l})^2}\pi_l(\theta)d\theta}
	{\int_\R e^{-\frac{n}{2}(\vartheta - f_{0,l})^2}\pi_l(\vartheta)d\vartheta },
\]
using the Laplace method for integral approximation, e.g.~\cite{breitung2006asymptotic}. See the proof of Lemma \ref{Lem:DetPart}.

\begin{remark}[Extension to other basis and metrics]
In the above steps, we restricted to wavelet bases and Sobolev norms for concreteness, motivated by our focus on Besov-Laplace priors (which are wavelet-based, cf.~Section \ref{Sec:LaplPriors}) and the applications to derivative estimation (cf.~Section \ref{Sec:DerivEstim}). However, the argument could also be repeated for any other orthonormal basis of $L^2$ and for any metric that may be represented as a series of coefficients in the considered basis.
\end{remark}

%
%
%
%
%

\section{Main results}
\label{Sec:Main}

%
%
%

\subsection{Besov-Laplace priors}
\label{Sec:LaplPriors}

Besov-Laplace priors represent a particular sub-class within the popular family of Besov priors introduced by Lassas et al.~\cite{LSS09}, and have been routinely employed as natural models for spatially inhomogeneous functions in a wide range of applications, including inverse problems \cite{leporini2001bayesian,DHS12,KLNS12} and imaging \cite{bioucas2006bayesian,niinimaki2007bayesian,sakhaee2015spline}. Such priors are defined as random wavelet expansions with independent coefficients following scaled Laplace distributions: let $\{\psi_l\}_{l=1}^\infty$ be an orthonormal wavelet basis of $L^2$, cf.~Section \ref{Sec:Notation}, and for any $\alpha>0$ define the \textit{$\alpha$-regular Besov-Laplace prior} $\Pi_\alpha$ on $L^2$ as the law of the random function \eqref{Eq:ProdPrior}, where
\beq
\label{Eq:LaplPrior}
	f_l\ind \pi_l\equiv \pi_{\alpha,l}, \qquad \pi_{\alpha,l}(\theta) 
	= \frac{l^{1/2+\alpha/d}}{2}e^{- l^{1/2+\alpha/d}|\theta|},\qquad \theta\in\R.
\eeq

	In other words, under $\Pi_\alpha$, each wavelet coefficient $f_l$ is distributed as the random variable $l^{-1/2-\alpha/d}Z_l$, where $Z_l\iid \text{Laplace}$. Since $(l^{-1/2-\alpha/d})_{l=1}^\infty\in\ell^2$ for all $\alpha>0$, the support of $\Pi_\alpha$ can be shown to be equal to $L^2$; further, the parameter $\alpha$ determines the regularity of the realisations of the draws from $\Pi_\alpha$: in particular, $\Pi_\alpha(B^s_p)=1$ for all $s<\alpha$ and all $p\ge 1$, while $\Pi_\alpha(B^s_p)=0$ whenever $s\ge\alpha$, for any $p$. See e.g.~\cite[Section 5.1]{ADH21} for details. Note that in the terminology of \cite{LSS09}, $\Pi_\alpha$ is called a $B^{\alpha+d}_{11}$-prior.

	Given observations $X^{(n)}\sim P_f^{(n)}$ from the white noise model \eqref{Eq:WhiteNoise}, the posterior $\Pi_\alpha(\cdot|X^{(n)})$ of $f|X^{(n)}$ arising from the above $\alpha$-regular Besov-Laplace prior $\Pi_\alpha$ follows as illustrated in Section \ref{Sec:ModelPriorPost}, and may be seen to assign to the wavelet coefficients the independent marginal distributions
\beq
\label{Eq:LaplPost}
	f_l|X^{(n)} \ind \pi_{n,l}(\cdot|X^{(n)}_l)\equiv\pi_{\alpha,n,l}(\cdot|X^{(n)}_l),
	\quad
	\pi_{\alpha,n,l}(\theta|x)
	=
	\frac{e^{-\frac{n}{2}(\theta - x)^2
	-l^{1/2+\alpha/d}|\theta|}}
	{\int_\R e^{-\frac{n}{2}(\vartheta - x)^2-l^{1/2+\alpha/d}|\vartheta|}
	 d\vartheta},
\eeq
where $\theta,x\in\R$.

%
%
%

\subsection{Strong posterior contraction rates for Besov-Laplace priors}
\label{Sec:StrongPCR}

We are now in a position to state and prove our main result concerning the asymptotic concentration, in $H^s$-metric, of the posterior distribution around the ground truth.

\begin{theorem}\label{Theo:Main}
Assume observations $X^{(n)}\sim P^{(n)}_{f_0}$ from the white noise model \eqref{Eq:WhiteNoise} for some fixed $f_0\in B^\beta_1$, some $\beta> d/2$. Let $\Pi_\beta(\cdot|X^{(n)})$ be the posterior distribution arising from a $\beta$-regular Besov-Laplace prior, defined as in \eqref{Eq:LaplPost}  with $\alpha=\beta$. Then, for all $ s\in[0, \beta-d/2)$,
\[
	E_{f_0}^{(n)}\Pi_\beta\left(f : \|f - f_0\|_{H^s}> M_nn^{-(\beta - s)/(2\beta + d)}
	\Big|X^{(n)}\right)\to 0
\]
as $n\to\infty$ for all positive real sequences $M_n\to\infty$.
\end{theorem}

\begin{proof}
From the general argument outlined in Section \ref{Sec:WassDyn}, it is sufficient to individually upper bound the two series appearing in the right hand side of \eqref{Eq:TwoSeries}, with marginal posterior distributions $\pi_{n,l}(\cdot|X^{(n)}_l) = \pi_{\beta,n,l}(\cdot|X^{(n)}_l)$ arising as in \eqref{Eq:LaplPost} with $\alpha=\beta$, and for wavelet coefficients $f_{0,l} = \langle f_0,\psi_l\rangle_{L^2}$. The claim then follows from the joint application of Lemmas \ref{Lem:StochPart} and \ref{Lem:DetPart}.
\end{proof}

	For all $\beta>d/2$ and $s<\beta$, the rate $n^{-(\beta-s)/(2\beta + d)}$ is known to be the minimax-optimal rate of estimation in $H^s$-metric for $\beta$-Sobolev-regular ground truths, cf.~\cite{fischer2020sobolev}, and this conclusions may be extended to the case $f_0\in B^\beta_1$ as well as long as $s<\beta-d/2$; see e.g.~\cite[Chapter 10]{HKPT98} for the case $d=1$ and  $s=0$ (i.e.~the $L^2$-minimax rates). Since the rates of posterior contraction are known to be lower bounded by the minimax ones (through the universal existence of an estimator that converges at the same rate as the posterior, cf.~Chapters 8.1 and 8.4 in \cite{GvdV17}), the rate obtained in Theorem \ref{Theo:Main} is sharp.

	Theorem \ref{Theo:Main} improves upon and expands the existing literature in two directions: firstly, in the presence of spatially inhomogeneous ground truths $f_0\in B^\beta_1$, the $L^2$-rate attained in the white noise model in \cite[Theorem 5.8]{ADH21} with a $\beta$-regular Besov-Laplace prior is of the form $n^{-(\beta-c)/(2\beta - 2c +d)}$ for some $c>0$, and thus is `polynomially slower' than the minimax rate. As mentioned in Section \ref{Sec:Literature}, the sub-optimality originates from the complexity of the `sieves sets' associated to general Besov priors, which complicates the derivation of sharp metric entropy bounds for the verification of the  testing condition of \cite{GGvdV00}. This difficulty was overcome in \cite{ADH21} and in all the subsequent frequentist asymptotic investigations of Besov priors \cite{GR22,giordano23besov,agapiou2024laplace,agapiou2024adaptive} through a combination of under-smoothing and rescaling, whereby in a variety of statistical models, for ground truths $f_0\in B^\beta_1$, optimal posterior contraction rates were obtained employing $(\beta-d)$-regular Besov priors, rescaled by a sequence that diverges as a power of the sample size. Setting $s=0$ in Theorem \ref{Theo:Main} shows that the smoothness-matching $\beta$-regular Laplace prior does in fact achieve the $L^2$-minimax rate $n^{-\beta/(2\beta+d)}$, in the white noise model. The proof relies on the Wasserstein dynamics approach outlined in Section \ref{Sec:WassDyn}, which is carried out via a careful study of the asymptotics behaviour of the individual terms appearing in the two series in the right hand side of \eqref{Eq:TwoSeries}. Crucially, this circumvents the need of employing the insufficiently sharp complexity bounds available for the sieves of the $\beta$-regular Laplace prior, which, as discussed above, result in a sub-optimal posterior contraction rate through the testing-based proof pursued in \cite{ADH21}.

	Secondly, Theorem \ref{Theo:Main} also expands the existing literature on Besov priors in terms of the metric in which the posterior contraction rates are formulated, allowing for Sobolev norms $\|\cdot\|_{H^s}$ of all orders $s<\beta-d/2$ (a natural restriction in view of the compact embedding $B^\beta_1\subset H^s$ holding for all $\beta>s+d/2$, e.g.~\cite[p.33]{LSS09}). On the contrary, all the other available results restrict to zeroth-order integral metrics of $L^p$-type, for some $p\in[1,2]$, see \cite{GR22,giordano23besov,agapiou2024laplace,agapiou2024adaptive}. In Section \ref{Sec:DerivEstim}, we will leverage the strong posterior contraction rates derived in Theorem \ref{Theo:Main} to show that smoothness-matching Laplace priors also simultaneously solve, with minimax-optimal posterior contraction rates, the problem of inferring the derivatives of the ground truth from observations from the white noise model \eqref{Eq:WhiteNoise}.

\begin{remark}[The case of $B^1_1$]
The space $B^1_1$ (the so-called `bump algebra') is closely related to the space of bounded variation functions, e.g.~\cite[Section 1]{DJ98}, and thus is of great relevance to a number of applications such as imagining and signal processing. Note that the case $f_0\in B^1_1$ is allowed in Theorem \ref{Theo:Main}, if $d=1$. On the contrary, all the existing results on minimax rates for Besov priors (which require undersmoothing and rescaling) do not strictly include this case, only covering regularities $\beta$ arbitrarily close to $1$, when $d=1$, see \cite{ADH21,giordano23besov,agapiou2024adaptive}.
\end{remark}



%
%
%

\subsection{Applications to derivative estimation}
\label{Sec:DerivEstim}

Given observations $X^{(n)}$ from the white noise model, we now turn to the problem of estimating the derivatives of the unknown regression function $f_0$. For any multi-index $\gamma = (\gamma_1,\dots,\gamma_d)\in \N_0^d$, let $|\gamma| := \sum_{j=1}^d \gamma_j$, and denote by
$$
	D^\gamma := \frac{\partial^{|\gamma|}}{\partial_{t_1}^{\gamma_1}\dots \partial_{t_d}^{\gamma_d}}
$$
the mixed partial weak differential operator of order $|\gamma|$. $D^\gamma$ is well-known to define a continuous linear operator between the Besov spaces $B^{s}_{pq}$ and $B^{s-|\gamma|}_{pq}$ for any $s$, $p$ and $q$ (e.g.~\cite[Chapter 4.3]{GN16}).

	For $\beta>|\gamma|$, assume that $f_0\in B^\beta_1$, and let $\Pi_\beta(\cdot|X^{(n)})$ be the posterior distribution considered in Section \ref{Sec:StrongPCR}, arising from endowing $f$ with a smoothness-matching $\beta$-regular Laplace prior. The so-called `plug-in' approach to estimating the mixed partial derivative $D^\gamma f_0$ consists in employing the posterior distribution $\Pi^{D^\gamma}_\beta(\cdot|X^{(n)})$ of $D^\gamma f|X^{(n)}$, obtained as the push-forward of $\Pi_\beta(\cdot|X^{(n)})$ under the differential operator $D^\gamma$. Compared to direct prior modelling of $D^\gamma f_0$, this is practically appealing, as it allows to employ a single statistical procedure (in particular, to specify a single prior distribution) regardless of wether the inferential target is the regression function or its derivatives. In view of the random series structure of $\Pi_\beta(\cdot|X^{(n)})$ (cf.~\eqref{Eq:LaplPost}) and the linearity of $D^\gamma$, the plug-in posterior $\Pi^{D^\gamma}_\beta(\cdot|X^{(n)})$ may also be seen to correspond to the law of a random series as in \eqref{Eq:ProdPrior}, with independent marginal posterior distributions $\pi_{\beta,n,l}(\cdot|X^{(n)}_l)$ as in \eqref{Eq:LaplPost}, and where the wavelet basis $\{\psi_l\}_{l\ge1}$ for the definition of $\Pi_\beta(\cdot|X^{(n)})$ is replaced by the collection $\{D^\gamma \psi_l\}_{l\ge1}$.
	


	The following corollary, which is a direct consequence of Theorem \ref{Theo:Main}, shows that the plug-in posteriors contract around the partial derivatives of $f_0$ at optimal $L^2$-rate. Such `plug-in property' was recently obtained for Gaussian process priors in nonparametric regression by Liu and Li \cite{liu2022optimal}, assuming Sobolev-regular ground truths. Here, it is established for Besov-Laplace priors and for regression functions in the Besov scale, representing, to our kwoledge, the first optimality result for Bayesian inference on (possibly) spatially-inhomogeneous derivatives.

\begin{corollary}\label{Cor:DerivEstim}
Assume observations $X^{(n)}\sim P^{(n)}_{f_0}$ from \eqref{Eq:WhiteNoise} for some fixed $f_0\in B^\beta_1$, some $\beta> 1+d/2$. For any $\gamma\in\N_0^d$ such that $|\gamma| < \beta - d/2$, let $\Pi^{D^\gamma}_\beta(\cdot|X^{(n)})$ be the push-forward under the operator $D^\gamma$ of the posterior distribution $\Pi_\beta(\cdot|X^{(n)})$, arising as in \eqref{Eq:LaplPost} with $\alpha=\beta$. Then, 
\[
	E_{f_0}^{(n)}\Pi^{D^\gamma}_\beta\left(g : \|g - D^\gamma f_0\|_{L^2}>
	M_n n^{-(\beta - |\gamma|)/(2\beta + d)}
	\Big|X^{(n)}\right)\to 0
\]
as $n\to\infty$ for all positive real sequences $M_n\to\infty$. 
\end{corollary}

\begin{proof}
Applying Theorem \ref{Theo:Main} with $s=|\gamma|\in[0,\beta-d/2)$ yields the posterior contraction rate $n^{-(\beta - |\gamma|)/(2\beta + d)}$ in $H^{|\gamma|}$-metric. In view of the norm equivalence (e.g.~\cite[p.~371]{GN16})
$$
	\|h\|_{H^{|\gamma|}}
	\simeq \sum_{\eta\in\N_0^d : |\eta|\le|\gamma|}\|D^\eta h\|_{L^2},
	\qquad h\in H^{|\gamma|},
$$
we have $\|D^\gamma h\|_{L^2}\lesssim \|h\|_{H^{|\gamma|}}$ for all $h\in H^{|\gamma|}$, and hence for some constant $c>0$ and all positive real sequences $M_n\to\infty$
\begin{align*}
	E_{f_0}^{(n)}&
	\Pi^{D^\gamma}_\beta\left(g : \|g - D^\gamma f_0\|_{L^2}> M_n 	n^{-(\beta - |\gamma|)/(2\beta + d)}
	\Big|X^{(n)}\right)\\
	&=
	E_{f_0}^{(n)}\Pi_\beta\left(f : \|D^\gamma f - D^\gamma f_0\|_{L^2}> M_n 	n^{-(\beta - |\gamma|)/(2\beta + d)}
	\Big|X^{(n)}\right)\\
	&\le E_{f_0}^{(n)}\Pi_\beta\left(f : \|f -  f_0\|_{H^{|\gamma|}}> c M_n 	n^{-(\beta - |\gamma|)/(2\beta + d)}
	\Big|X^{(n)}\right)\to 0.
\end{align*}
\end{proof}

%
%
%
%
%

\section{Discussion}
\label{Sec:Discussion}

	In this article, we have investigated the recovery of (possibly) spatially inhomogeneous grounds truths and their derivatives in the white noise model using Besov-Laplace priors. We have shown that smoothness-matching priors attains minimax optimal posterior contraction rates over the Besov spaces $B^\beta_1$, $\beta > d/2$, not requiring artificial undersmoothing and rescaling as in the previously available literature. The proof is based on the novel `Wasserstein dynamics' approach to posterior contraction rates proposed in \cite{dolera2024strong}, which allows to overcome the technical challenges that arise in the pursuit of the testing-based strategy of \cite{GGvdV00} for Besov-Laplace priors of matching regularity. The employed proof technique also allow to flexibly handle (stronger) Sobolev metrics, which we exploit to show that the `plug-in' posteriors under differential operators optimally estimate the derivatives of the ground truth.

%
%
%

\subsection{Adaptation}

The posterior contraction rates obtained in Theorem \ref{Theo:Main} and Corollary \ref{Cor:DerivEstim} are non-adaptive, since they are obtained with smoothness-matching prior distributions, whose specification requires knowledge of the regularity of the ground truth. Hierarchical procedures based on undersmoothing and rescaled Besov priors have been recently shown to achieve adaptation over the Besov scale in \cite{agapiou2024adaptive}, which requires the randomisation of two hyperparameters, as well as \cite{giordano23besov}, where an `adaptive rescaling' is employed.

	The present investigation suggests that simpler hierarchical Besov priors, constructed along the lines of those in \cite{vdVvZ09} or \cite{lember2007universal}, could also achieve adaptive posterior contraction rates, as well as an `adaptive plug-in' property for the estimation of the derivatives (cf.~\cite{liu2022optimal}, where the latter property is obtained for ground truths in the Sobolev scale with an empirical Bayes procedure based on Gaussian process priors). However, since hierarchical priors generally fall outside the product framework outlined in Section \ref{Sec:WassDyn}, the application of the Wasserstein dynamics approach to posterior contraction rates in such context would result in a substantially different mathematical analysis that the one developed here. At the present stage, we thus leave this question for future investigations.

%
%
%

\subsection{Beyond Besov-Laplace priors and the white noise model}
	
Another interesting research direction concerns the extension of the present work to other classes of prior distributions and different statistical models, in particular to settings where the verification of the testing condition of \cite{GGvdV00} for posterior contraction may prove to be challenging. For such cases, our results, alongside the investigations of \cite{dolera2024strong}, show that the Wasserstein dynamics approach may provide a useful alternative. For a related work with a similar point of view, we refer to the recent article by Agapiou and Castillo \cite{agapiou2024heavy}, where adaptive posterior contraction rates in the white noise model are derived (with different techniques) using product priors with heavy tails (heavier than Laplace), not necessitating any hyperparameter tuning.

	In \cite{dolera2024strong}, the application of the Wasserstein dynamics approach to posterior contraction rates was examined in several prototypical settings, including parametric statistical models, nonparametric regression and density estimation. However, while \cite{dolera2024strong} laid down the foundational work for these applications, the optimality of their results (outside of those in the parametric framework) remains unclear. These issues are the subject of ongoing work, but they are outside the scope of the present article.

%
%
%
%
%

\section{Proofs}
\label{Sec:Proofs}

In this section we will state and prove the two main technical tools, Lemmas \ref{Lem:StochPart} and \ref{Lem:DetPart}, for the proof of Theorem \ref{Theo:Main}, providing upper bounds for the two series appearing in the right hand side of \eqref{Eq:TwoSeries} for smoothness-matching Besov-Laplace priors. To unburden the notation, we will denote  throughout, for any fixed $\beta>d/2$, $\gamma_l \equiv \gamma_{l,\beta}:=\ell^{1/2+\beta/d}$, and rewrite the marginal posteriors from \eqref{Eq:LaplPost} as
\beq\label{Eq:NewPost}
	\pi_{\beta,n,l}(\theta|x) = \frac{e^{-\frac{n}{2}(\theta - x)^2-\gamma_l|\theta|}}
	{\int_\R e^{-\frac{n}{2}(\vartheta - x)^2-\gamma_l|\vartheta|}d\vartheta},
	\qquad \theta,x\in\R.
\eeq

%
%
%

\subsection{Stochastic term}

Given observations $X^{(n)}\sim P^{(n)}_{f_0}$ for some $f_0\in B^\beta_1$, and the posterior distribution $\Pi_\beta(\cdot|X^{(n)})$ arising from a $\beta$-regular Besov-Laplace prior, the first series in \eqref{Eq:TwoSeries} reads, for any $s<\beta-d/2$,
\beq\label{Eq:NewSeries}
	\sum_{l=1}^\infty l^{2s/d}E^{(n)}_{f_{0}}
	\left[ \Wcal_2^2\left(\pi_{\beta,n,l}(\cdot| X^{(n)}_l), \pi_{\beta,n,l}(\cdot| f_{0,l})\right) \right],
\eeq
with $\pi_{\beta,n,l}(\cdot| X^{(n)}_l)$ and $\pi_{\beta,n,l}(\cdot| f_{0,l})$ defined as in \eqref{Eq:NewPost} with $x=X^{(n)}_l$ and $x=f_{0,l}=\langle f_0,\psi_l\rangle_{L^2}$, respectively. As outlined in Section \ref{Sec:WassDyn}, we will bound the individual terms in the above series by deriving a $\Wcal_2$-Lipschitz property, cf.~\eqref{Eq:LipProp}, for the Markov kernels $\{\pi_{\beta,n,l}(\cdot|x),\ x\in\R\}$ with respect to their second argument. This idea, which is akin to the notion of `continuity of the posterior distribution with respect to the data' from the inverse problem literature \cite{S10} (where typically the Hellinger distance is employed), was recently explored by Dolera and Mainini \cite{dolera2020uniform,dolera2023lipschitz} and linked to the concepts of functional inequalities for probability measures and weighted Poincaré constants.

	The following lemma, which is of independent interest, shows that the Lipschitz property in $\Wcal_2$-distance \eqref{Eq:LipProp} holds for the marginal posterior distributions arising from Besov-Laplace priors, and provide precise estimates for the Lipschitz constants with explicit dependence on $l,n$ and $\beta$. Its proof makes use of the main result from \cite{dolera2023lipschitz}, and is based upon separating in the series \eqref{Eq:NewSeries} the `low' frequencies $l\le L_n$ from the `high' ones $l>L_n$, in terms of a threshold $L_n\in\N$ such that $\gamma_l = l^{1/2+\beta/d}\lesssim \sqrt{n}$ if $l\le L_n$, and $\gamma_l\gtrsim \sqrt n$ for all $l>L_n$. In the first regime, the posterior densities from \eqref{Eq:NewPost} have the form of `small' log-concave perturbations of a Gaussian measure with variance $ 1/n$, cf.~\eqref{Eq:Pert} below, from which they inherit an estimate of the associated Poincaré constants \cite{cattiaux2022functional,bardet18}. For the high frequencies, the posterior densities instead may be be interpreted as `small' Gaussian perturbations of scaled Laplace distributions, cf.~\eqref{Eq:Pert2} below, which also allows some explicit manipulations to obtain the required estimates of the Poincaré constants using results from \cite{cattiaux2022functional}.

\begin{lemma}\label{Lem:LipProp}
For any $\beta>d/2$, $n,l\in\N$, let $\{\pi_{\beta,n,l}(\cdot|x)\}$ be the Markov kernel defined in \eqref{Eq:NewPost}. Let $L_n\in \N$ be the smallest integer number satisfying $L_n^{(2\beta+d)/d}\ge 2n$. Then, for all $l\le L_n$ and all $x,y\in\R$, we have
\[
	\Wcal_2(\pi_{\beta,n,l}(\cdot|x),\pi_{\beta,n,l}(\cdot|y))
	\le 4e^{16\sqrt{2/\pi}}|x - y|.
\]
Further, for all $l>L_n$ and all $x,y\in\R$, we have 
\[
	\Wcal_2(\pi_{\beta,n,l}(\cdot|x),\pi_{\beta,n,l}(\cdot|y))
	\le 8\frac{n}{\gamma_l^2}|x-y|.
\]
\end{lemma}

\begin{proof}
We apply Corollary 2.3 in \cite{dolera2023lipschitz} (with the choices, in their notation, of $f(x|\theta) := e^{-\frac{n}{2}(\theta - x)^2}$, for which $\Phi(x,\theta)= -n(\theta - x)^2/2$, $\partial_x \Phi(x,\theta) = n(\theta - x)$ and  $\partial_\theta\partial_x \Phi(x,\theta) = n$) to obtain that for all $n,l\in\N$, $x,y\in\R$,
\beq\label{Eq:WassLip}
	\Wcal_2\left(\pi_{\beta,n,l}(\cdot| x), \pi_{\beta,n,l}(\cdot| y)\right)
	\le n\sup_{z\in\R}\{\Ccal[\pi_{\beta,n,l}(\cdot|z))]^2\}| x - y|,
\eeq
where $\Ccal[\pi_{\beta,n,l}(\cdot|x))]$ is the Poincaré-constant associated to the probability measure $\pi_{\beta,n,l}(\cdot|x)$, $x\in\R$; see e.g.~eq.~(2.2) in \cite{dolera2023lipschitz} for the definition. We proceed upper bounding the constants $\Ccal[\pi_{\beta,n,l}(\cdot|x))]^2$, separating the cases $l\le L_n$ and $l>L_n$.

	Starting with the low frequencies $l\le L_n$, for which have $\gamma_l^2/n\le4$, we rewrite the posterior densities from \eqref{Eq:NewPost} as
\beq\label{Eq:Pert}
	\pi_{\beta,n,l}(\theta|x)
	\equiv \mu_F(\theta)
	:=\frac{e^{-F(\theta)}\mu(\theta)}
	{\int_\R e^{-F(\vartheta)}\mu(\vartheta)d\vartheta}
	\qquad \theta,x\in\R.
\eeq	
where 
\[
	F(\theta) := \gamma_l |\theta|, \qquad \theta\in\R,
\]
and
\[
	\mu(\theta) := \frac{\sqrt n}{\sqrt{2\pi} }e^{-\frac{n}{2}(\theta - x)^2} 
	\propto e^{-V(\theta)}, 
	\qquad V(\theta):=\frac{n}{2}(\theta - x)^2,
	\qquad \theta\in\R.
\]
Noting that $V''(\theta) = n/2$ (constant in $x$) and that $F$ is a Lipschitz continuous function with Lipschitz constant equal to $\gamma_l$, Theorem 1.3 in \cite{cattiaux2022functional} (with the choices, in their notation, $\rho: = n/2$, and $L:=\gamma_l$), see also Lemma 2.1 in \cite{bardet18}, then implies that for all $x\in\R$ ,
\[
	\Ccal[\pi_{\beta,n,l}(\cdot|x))]^2 \le \frac{2}{ n/2} e^{4\sqrt{2/\pi}\gamma_l/\rho}
	\le 4e^{16\sqrt{2/\pi}}\frac{1}{n}.
\]
Since the latter is independent of $x$, the first statement then follows from \eqref{Eq:WassLip}. Moving to the case $l>L_n$, we follow Example 2 in \cite{cattiaux2022functional}, now writing
\beq\label{Eq:Pert2}
	\pi_{\beta,n,l}(\theta|x)\propto e^{-V(\theta) - \rho (\theta - x)^2/2 }, 
	\qquad
	V(\theta) := \gamma_l|\theta|,
	 \qquad \rho := n,
	\qquad \theta\in\R.
\eeq
Note that $\rho$ above is independent of $x$. We then have $V'(\cdot) = \sign(\cdot)$ (in the weak sense) and 
\[
 	\theta V'(\theta) = x \gamma_l \sign(\theta) = \gamma_l |\theta| > - K - K'|\theta|^2
\]
for all $K,K'\ge 0$, all $\theta\in\R$. It follows that for all $x\in \R$,
\begin{equation}
\label{Eq:ShiftPoinc}
	\Ccal[\pi_{\beta,n,l}(\cdot|x))]^2 \le \frac{\Ccal[\pi_{l}(\cdot)]^2}{1-\epsilon},
\end{equation}
for all $\epsilon>0$ such that
\begin{equation}
\label{Eq:EpsCond}
	\frac{2(1-\epsilon)}{ \Ccal[\pi_{l}(\cdot)]^2 ( 1 + K) }\ge  \rho = n,
\end{equation}
cf.~Example 2 in \cite{cattiaux2022functional}. Above, $\pi_l(\theta) = \gamma_l e^{-\gamma_l|\theta|}/2$ is the marginal prior density, corresponding to a scaled Laplace distribution on $\R$. The Poincaré constant for the standard Laplace distribution, with density 
$\pi_1(\theta) = e^{-|\theta|}/2$, is well-known to satisfy $\Ccal[\pi_{1}(\cdot)]^2 = 4$ (cf.~Example 1 in \cite{cattiaux2022functional}), and we now argue that $\Ccal[\pi_{l}(\cdot)]^2 \le 4/\gamma_l^2$, namely that for any differentiable test function $\phi:\R\to\R$ (for which the following integrals are well defined),
\begin{equation}
\label{Eq:LaplPoinc}
	\int_\R (\phi(\theta) - \pi_l[\phi])^2\pi_l(\theta)d\theta \le \frac{4}{\gamma_l^2}
	\int_\R [\phi'(\theta)]^2\pi_l(\theta)d\theta,
	\qquad \pi_l[\phi] := \int_\R \phi(\vartheta)\pi_l(\vartheta)d\vartheta.
\end{equation}
Indeed, using (twice) the change of variable $t = \gamma_l\theta$, the left hand side equals
\begin{align*}
	&\int_\R \left(\phi(\theta) 
	- \int_\R \phi(\vartheta)\frac{\gamma_l}{2} e^{-\gamma_l|\vartheta|}d\vartheta
	\right)^2
	\frac{\gamma_l}{2} e^{-\gamma_l|\theta|}d\theta\\
	&\ =
	\int_\R \left(\phi(t/\gamma_l) - \int_\R \phi(\tau/\gamma_l)\frac{1}{2} e^{-|\tau|}d\tau
	\right)^2
	\frac{1}{2} e^{-|t|}dt
	=
	\int_\R \left(\phi(t/\gamma_l) - \pi_1[\phi(\cdot /\gamma_l)]\right)^2
	\pi_1(t)dt,
\end{align*}
which, using that $\Ccal[\pi_{1}(\cdot)]^2 = 4$, is upper bounded by
\begin{align*}
	4 \int_\R \left[\frac{d}{d \tau }\phi(\tau/\gamma_l)\Big|_{\tau=t}\right]^2\pi_1(t)dt
	&=\frac{4}{\gamma_l^2} \int_\R \left[\phi'(t/\gamma_l)\right]^2\frac{\gamma_l}{2}e^{-|t|}\frac{1}{\gamma_l}dt.
\end{align*}
With the substitution $\theta = t/\gamma_l$, the last integral is seen to be equal to 
\[
	\frac{4}{\gamma_l^2} \int_\R \left[\phi'(\theta)\right]^2\frac{\gamma_l}{2}e^{-\gamma_l|\theta|}d\theta
	=\frac{4}{\gamma_l^2} \int_\R \left[\phi'(\theta)\right]^2\pi_l(\theta)d\theta,
\]
concluding the derivation of \eqref{Eq:LaplPoinc}.  Thus, recalling that for all $l>L_n$ we have $\gamma_l^2 \ge 2n$, taking $\epsilon = 1/2$ (and $K=K'=0$) in \eqref{Eq:EpsCond} yields as required that
\[
	 \frac{2(1-\epsilon)}{ \Ccal[\pi_{l}(\cdot)]^2  } = \frac{1}{2}\gamma_l^2 \ge n,
\]
and, in view of \eqref{Eq:ShiftPoinc}, that
\[
	\Ccal[\pi_{\beta,n,l}(\cdot|x))]^2 \le 2\Ccal[\pi_{l}(\cdot)]^2 \le \frac{8}{\gamma_l^2}.
\]
The second claim then follows from \eqref{Eq:WassLip}, as the latter is independent of $x$. 

\end{proof}

	The next lemma leverages the Lipschitz properties derived in Lemma \ref{Lem:LipProp} to upper bound the series of expected Wasserstein distances in \eqref{Eq:NewSeries}

\begin{lemma}\label{Lem:StochPart}
Assume observations $X^{(n)}\sim P^{(n)}_{f_0}$ from the white noise model \eqref{Eq:WhiteNoise} for some fixed $f_0\in B^\beta_1$, some $\beta> d/2$. Let $\Pi_\beta(\cdot|X^{(n)})$ be the posterior distribution arising from a $\beta$-regular Besov-Laplace prior, with marginal posterior distributions $\pi_{\beta,n,l}(\cdot|X^{(n)})$, $l\in\N$, defined as in \eqref{Eq:LaplPost}  with $\alpha=\beta$. Then, for all $ s\in[0, \beta-d/2)$ and all sufficiently large $n$,
\[
	\sum_{l=1}^\infty l^{2s/d}E^{(n)}_{f_{0}}
	\left[ \Wcal_2^2\left(\pi_{\beta,n,l}(\cdot| X^{(n)}_l) \pi_{\beta,n,l}(\cdot| f_{0,l})\right) \right]
	\lesssim n^{-2(\beta-s)/(2\beta+d)}.
\]
\end{lemma}

\begin{proof}
Split the series under study at the threshold $L_n\in\N$ defined in the statement of Lemma \ref{Lem:LipProp}. Then, by the first statement of Lemma \ref{Lem:LipProp},
\begin{align*}
	\sum_{l=1}^{L_n}  l^{2s/d}E^{(n)}_{f_{0}}
	\left[ \Wcal_2^2\left(\pi_{\beta,n,l}(\cdot| X^{(n)}_l) \pi_{\beta,n,l}(\cdot| f_{0,l})\right) \right]
	&\lesssim 
	\sum_{l=1}^{L_n}   l^{2s/d}E^{(n)}_{f_{0}}
	[| X^{(n)}_l - f_{0,l}|^2 ]\\
	&\le \frac{1}{n}\sum_{l=1}^{L_n}   l^{2s/d},
\end{align*}
having used that $X^{(n)}_l\sim N(f_{0,l},1/n)$ under $P^{(n)}_{f_{0}}$, cf.~\eqref{Eq:GaussSeq}. Using the Euler-Maclaurin method, the last quantity equals, for $(B_k)_{k\ge1}$ the sequence of Bernoulli's numbers (satisfying $B_k\to0$ as $k\to\infty$),
\begin{align*}
	\frac{1}{ n} \Bigg\{ \int_1^{L_n}&l^{2s/d}dl 
	+ \sum_{k=1}^\infty \frac{B_k}{k!}\times\frac{d^{k-1}}{dl^{k-1}}l^{2s/d}\Big|^{L_n}_{1}
	\Bigg\}\\
	&=\frac{1}{ n}\Bigg\{\frac{L_n^{2s/d+1}-1}{2s/d+1}
	+\frac{L_n^{2s/d}-1}{2}
	+\sum_{k=2}^\infty \frac{B_k}{k!}\times\frac{d^{k-1}}{dl^{k-1}}l^{s/d}\Big|^{L_n}_{1}\Bigg\}\\
	&\lesssim  \frac{1}{n}L_n^{2s/d+1}
	\lesssim \frac{1}{n} n^{(2s+d)/(2\beta+d)}
	=n^{(2s + d - 2\beta-d)/(2\beta + d)} = n^{-2(\beta - s)/(2\beta + d)},
\end{align*}
the first inequality in the last line holding for all sufficiently large $n\in\N$ since $L_n\to\infty$ as $n\to\infty$. Next, by the second statement of Lemma \ref{Lem:LipProp},
\begin{align*}
	\sum_{l=L_n+1}^\infty & l^{2s/d}E^{(n)}_{f_{0}}
	\left[ \Wcal_2^2\left(\pi_{\beta,n,l}(\cdot| X^{(n)}_l) \pi_{\beta,n,l}(\cdot| f_{0,l})\right) \right]\\
	 &\lesssim n^2\sum_{l = L_n+1}^\infty \frac{l^{2s/d}}{\gamma_l^4}
	 E^{(n)}_{f_{0}}
	[| X^{(n)}_l - f_{0,l}|^2 ]\\
	 &=n\sum_{l = L_n+1}^\infty  l^{-(4\beta-2s+2d)/d}
         =n L_n^{-\frac{4\beta - 2s + d}{d}}
         \lesssim n n^{-\frac{4\beta - 2s + d}{2\beta + d}}
         = n^{-\frac{2(\beta-s)}{2\beta + d}}.
\end{align*}
Combined with the previous display, this completes the proof.
\end{proof}

%
%
%

\subsection{Deterministic term}

To conclude the proof of Theorem \ref{Theo:Main}, there remains to upper bound the second series in the right hand side of \eqref{Eq:TwoSeries}, which equals
\[
	\sum_{l=1}^\infty l^{2s/d}
	\Wcal_2^2\left( \pi_{\beta,n,l}(\cdot| f_{0,l}),\delta_{f_{0,l}}\right)
	=\sum_{l=1}^\infty l^{2s/d}
	\frac{\int_\R |\theta-f_{0,l}|^2e^{-\frac{n}{2}(\theta - f_{0,l})^2}
	e^{-|\gamma_l\theta|}d\vartheta}
	{\int_\R e^{-\frac{n}{2}(\vartheta - f_{0,l})^2} e^{-|\gamma_l\vartheta|}d\theta},
\]
where $\gamma_l=l^{1/2+\beta/d}$. As $n\to\infty$, the individual terms of the series in the right hand side have the form of a ratio between the integrals of two vanishing functions. We then proceed by deriving precise asymptotics for these terms, carefully keeping track of the dependence on $l$. In a similar spirit to the proof of Lemma \ref{Lem:StochPart}, we will split the series into `low', `intermediate' and `high' frequencies. For the first regime, suitable asymptotics can then be obtained via the Laplace method for integral approximation \cite{breitung2006asymptotic}. For the intermediate and high frequencies, we instead derive the required upper bounds via an ad hoc analysis, exploiting, among other things, the fact that the numerator of the marginal posterior distributions in \eqref{Eq:NewPost} can be expressed in terms of a certain special function, cf.~Lemma \ref{Lem:LaplPriorDenom}, whose asymptotics properties are explicitly available.

\begin{lemma}\label{Lem:DetPart}
Assume observations $X^{(n)}\sim P^{(n)}_{f_0}$ from the white noise model \eqref{Eq:WhiteNoise} for some fixed $f_0\in B^\beta_1$, some $\beta> d/2$. Let $\Pi_\beta(\cdot|X^{(n)})$ be the posterior distribution arising from a $\beta$-regular Besov-Laplace prior, with marginal posterior distributions $\pi_{\beta,n,l}(\cdot|X^{(n)})$, $l\in\N$, defined as in \eqref{Eq:LaplPost}  with $\alpha=\beta$. Then, for all $ s\in[0, \beta-d/2)$, and all positive sequences $M_n\to\infty$, we have for all sufficiently large $n$,
\[
	\sum_{l=1}^\infty l^{2s/d}
	\Wcal_2^2\left( \pi_{\beta,n,l}(\cdot| f_{0,l}),\delta_{f_{0,l}} \right)
	\lesssim M_nn^{-2(\beta-s)/(2\beta+d)},
\]
\end{lemma}

\begin{proof}

Using the change of variable $t=\gamma_l\theta$, the series of interest equals
\begin{equation}
\label{Eq:LaplPriorStart}
\begin{split}
	\sum_{l=1}^\infty l^{2s/d}&
	\frac{\int_\R \gamma_l^{-2} (\gamma_l\theta-\gamma_lf_{0,l})^2e^{-\frac{n}{2\gamma_l^2}
	(\gamma_l\theta - \gamma_l f_{0,l})^2}e^{-|\gamma_l\theta|}d\theta}
	{\int_\R e^{-\frac{n}{2\gamma_l^2}(\gamma_l \vartheta - \gamma_lf_{0,l})^2} 
	e^{-|\gamma_l\vartheta|}d\vartheta}\\
	&\qquad\qquad\qquad\qquad\qquad\qquad=\sum_{l=1}^\infty l^{2s/d}\gamma_l^{-2}
	\frac{\int_\R  (t-t_{0,l})^2e^{-\frac{n}{2\gamma_l^2}
	(t - t_{0,l})^2}e^{-|t|}dt}
	{\int_\R e^{-\frac{n}{2\gamma_l^2}(\tau - t_{0,l})^2}
	 e^{-|\tau|}d\tau},
\end{split}
\end{equation}
having denoted $t_{0,l}:=\gamma_lf_{0,l}$, $l\in\N$. Note that, since $f_0\in b^\beta_1$ by assumption, we have
\beq\label{Eq:BesovDecay}
	\sum_{l=1}^\infty l^{-1}|t_{0,l}| = \sum_{l=1}^\infty  l^{-1}\gamma_l|f_{0,l}| 
	= \sum_{l=1}^\infty l^{\beta/d-1/2}|f_{0,l}|=\|f_0\|_{B^\beta_1}<\infty,
\eeq
whence $t_0:=(t_{0,l})_{l=1}^\infty$ is seen to necessarily be in $\ell^\infty_0$ (the space of bounded sequences that vanish asymptotically). Since $M_n\to\infty$ is an arbitrarily slow diverging sequence, we can assume without loss of generality that $M_n\le \log n$. We proceed by splitting the series in the right hand side of \eqref{Eq:LaplPriorStart} at levels $J_n, L_n\in\N$ satisfying $J_n\simeq (n/M_n)^{d/(2\beta+d)}$ and $L_n\simeq n^{d/(2\beta+d)}$. Note that for these, we have for $1\le l\le J_n$,
\begin{equation}
\label{Eq:GammaL1}
	\frac{n}{2\gamma_l^2} 
	\ge \frac{n}{2\gamma_{J_n}^2}
	= \frac{n}{2J_n^{1+2\beta/d}} 
	\simeq
	M_n  \to \infty
\end{equation}
for all $J_n< l\le L_n$, 
\begin{equation}
\label{Eq:GammaL2}
	1
	\simeq \frac{n}{2\gamma_{L_n}^2}
	\le
	 \frac{n}{2\gamma_l^2}
	 <
	\frac{n}{2\gamma_{J_n}^2}
	\simeq M_n,
\end{equation}
and for all $l>L_n$,
\begin{equation}
\label{Eq:GammaL3}
	 \frac{n}{2\gamma_l^2}
	<
	\frac{n}{2\gamma_{L_n}^2}
	 \simeq 1.
\end{equation}

	We bound the series for the low frequencies
\bal
\label{Eq:FirstSumL2}
	\sum_{l =1}^{J_n}&\gamma_l^{-2}l^{2s/d}
	\frac{\int_\R (t - t_{0,l})^2e^{-\frac{n}{2\gamma_l^2}(t - t_{0,l})^2}e^{-|t|}dt}
	{\int_\R e^{-\frac{n}{2\gamma_l^2}(t - t_{0,l})^2}e^{-|t|} dt },
\end{align}
via an application of the Laplace method for integral approximation. In particular, for all $l\le J_n$, recalling \eqref{Eq:GammaL1}, Theorem 41, p.~56, in \cite{breitung2006asymptotic} (with the choices, in their notation, $\beta^2:=n/(2\gamma_l^2)\to\infty$ and $x^*:=t_{0,l}$) implies that as $n\to\infty$, the denominator satisfies
\beq\label{Eq:AsympDenom}
	\int_\R e^{-|\tau|}e^{\frac{n}{2\gamma_l^2}[-(\tau - t_{0,l})^2]}d\tau
	\simeq 
	\sqrt{2\pi}\frac{e^{-|t_{0,l}|}}{\sqrt 2} \left( \frac{n}{ 2 \gamma_l^2}\right)^{-1/2}
	\simeq e^{-|t_{0,l}|}\frac{\gamma_l}{\sqrt n} .
\eeq
Further, in order to apply Theorem 43, p.~60, in \cite{breitung2006asymptotic} to the numerator, note that as $t\to t_{0,l}$, by a first order Taylor expansion,
$$
	e^{-|t|}
	=e^{-|t_{0,l}|} - e^{-|t_{0,l}|}\textnormal{sign}(t_0)(t-t_{0,l})+o(|t-t_{0,l}|),
$$
so that
$$
	(t-t_{0,l})^2e^{-|t|}=(t-t_{0,l})^2e^{-|t_{0,l}|} - e^{-|t_{0,l}|}\textnormal{sign}(t_0)(t-t_{0,l})^3+o(|t-t_{0,l}|^3).
$$
Therefore, Theorem 43, p.~60, in \cite{breitung2006asymptotic} (with the choices, in their notation, $\nu:=2$ and $h_0(t):=e^{-|t_{0,l}|}$, which is constant in $t$) implies that
\begin{align*}
	\int_\R  (t - t_{0,l})^2e^{-|t|}e^{-\frac{n}{2\gamma_l^2}(t - t_{0,l})^2}dt
	\simeq \sqrt{2}\Gamma(3/2)\frac{e^{-|t_{0,l}|}}{\sqrt 2}
	 \left(\frac{\sqrt n}{\sqrt{2}\gamma_l}\right)^{-3}
	\simeq e^{-|t_{0,l}|}\frac{\gamma_l^3}{n^{3/2}}.
\end{align*}
Combined with \eqref{Eq:AsympDenom}, this yields, via an application of the Euler-Maclaurin method,
\begin{align*}
	\sum_{l =1}^{J_n}&\gamma_l^{-2}l^{2s/d}
	\frac{\int_\R (t - t_{0,l})^2e^{-\frac{n}{2\gamma_l^2}(t - t_{0,l})^2}e^{-|t|}dt}
	{\int_\R e^{-\frac{n}{2\gamma_l^2}(t - t_{0,l})^2}e^{-|t|} dt }\\
	&\simeq 
	\frac{1}{n} \sum_{l =1}^{J_n}l^{2s/d}\\
	&\lesssim \frac{1}{n}J_n^{2s/d+1}
	\lesssim \frac{1}{n}(n/M_n)^{(2s+d)/(2\beta+d)}
	\le n^{(2s + d - 2\beta-d)/(2\beta + d)} = n^{-2(\beta - s)/(2\beta + d)}
\end{align*}
the first inequality in the last line holding for all sufficiently large $n\in\N$ since $J_n\to\infty$ as $n\to\infty$. Next, for the intermediate frequencies, using Lemma \ref{Lem:LaplPriorDenom} and recalling that $t_{0,l}=\gamma_l f_{0,l}$, we have
\beq\label{Eq:ErfcDenom}
\begin{split}
	\int_\R &e^{-\frac{n}{2\gamma_l^2}(\tau - t_{0,l})^2}
	 e^{-|\tau|}d\tau \\
	&\qquad\qquad =
	\frac{\sqrt{\pi}\gamma_l}{\sqrt {2n}} 
	e^{ \frac{\gamma_l^2}{2n}}
	\Bigg[e^{-t_{0,l}}
	\erfc\left(-\frac{\sqrt n}{\sqrt{2}}f_{0,l} + \frac{\gamma_l}{\sqrt {2n}}\right)
	+
	e^{t_{0,l}}
	\erfc\left(\frac{\sqrt n}{\sqrt{2}}f_{0,l} + \frac{\gamma_l}{\sqrt{2n}} 
	\right)
	\Bigg],
\end{split}
\eeq
where $\erfc$ is the complementary error function, cf.~\eqref{Eq:ERFC}. In view of \eqref{Eq:GammaL2}, we have $1/ \sqrt{M_n} \lesssim \gamma_l/\sqrt n \lesssim 1$ for all for all $J_n\le l\le L_n$, so that by the boundedness of $t_0\in\ell^\infty_0$ and the fact that $\erfc$ is a smooth, strictly decreasing function satisfying
$$
	\lim_{z\to-\infty}\erfc(z)=2;
	\qquad \lim_{z\to\infty}\erfc(z)=0,
$$
cf.~\cite[Chapter 7]{olver2010nist}, we may conclude that for all $J_n\le l\le L_n$ and all sufficiently large $n\in\N$,
$$
	e^{-t_{0,l}}
	\erfc\left(-\frac{\sqrt n}{\sqrt{2}}f_{0,l} + \frac{\gamma_l}{\sqrt {2n}}\right)
	+
	e^{t_{0,l}}
	\erfc\left(\frac{\sqrt n}{\sqrt{2}}f_{0,l} + \frac{\gamma_l}{\sqrt{2n}} 
	\right)
	\simeq 1.
$$
Combined with \eqref{Eq:ErfcDenom}, this implies
\begin{equation}
\label{Eq:erfcBound}
	\int_\R e^{-\frac{n}{2\gamma_l^2}(t - t_{0,l})^2}e^{-|t|} dt
	\gtrsim \frac{\gamma_l}{\sqrt n}\ge \frac{1}{\sqrt{M_n}},
\end{equation}
from which we obtain, for all sufficiently large $n\in\N$,
\begin{align*}
	\sum_{l =J_n+1}^{L_n}&\gamma_l^{-2}l^{2s/d}
	\frac{\int_\R (t - t_{0,l})^2e^{-\frac{n}{2\gamma_l^2}(t - t_{0,l})^2}e^{-|t|}dt}
	{\int_\R e^{-\frac{n}{2\gamma_l^2}(t - t_{0,l})^2}e^{-|t|} dt }\\
	&\lesssim
	\sqrt {M_n}\sum_{l =J_n+1}^{L_n}\gamma_l^{-2}l^{2s/d}
	\int_\R (t - t_{0,l})^2e^{-\frac{n}{2\gamma_l^2}(t - t_{0,l})^2}e^{-|t|}dt.
\end{align*}
Using again the boundedness of $t_0\in\ell^\infty_0$, and the fact that $e^{-n(t - t_{0,l})^2/({2\gamma_l^2})}\le 1$, the latter integrals are upper bounded (uniformly in $l$) by
\begin{align*}
	\int_\R t^2 e^{-|t|}dt + 2|t_{0,l}| \int_\R |t| e^{-|t|}dt + |t_{0,l}^2| \int_\R e^{-|t|}dt \lesssim 1
\end{align*}
so that by another application of the Euler-Maclaurin method, the second to last display is seen to be smaller than a multiple of
\begin{align*}
	&\sqrt{M_n}\sum_{l =J_n+1}^{L_n}l^{-2(\beta-s)/d-1}\\
	&\ =
	\sqrt{M_n}\left\{\int_{J_n+1}^{L_n}l^{-2(\beta-s)/d-1}dl 
	+ \sum_{k=1}^\infty \frac{B_k}{k!}\times\frac{d^{k-1}}{dl^{k-1}}l^{-2(\beta-s)/d-1}\Big|_{J_n+1}^{L_n}\right\}\\
	&\ =\sqrt{M_n}\Bigg\{\frac{(J_n+1)^{-2(\beta-s)/d}-L_n^{-2(\beta-s)/d}}{(\beta-s)/d}
	+\frac{(J_n+1)^{-2(\beta-s)/d-1}-L_n^{-2(\beta-s)/d-1}}{2}\\
	&\ \quad\ 
	+\sum_{k=2}^\infty \frac{B_k}{k!}\times\frac{d^{k-1}}{dl^{k-1}}l^{-2(\beta-s)/d-1}
	\Big|_{J_n+1}^{L_n}
	\Bigg\}\\
	&\ \lesssim
	\sqrt {M_n} J_n^{-2(\beta-s)/d}\\
	&\lesssim \sqrt {M_n} (n/M_n)^{-2(\beta-s)/(2\beta + d)}
	=n^{-2(\beta-s)/(2\beta+d)}
	M_n^{(3\beta-2s+d/2)/(2\beta+d)},
\end{align*}
the second to last inequality holds for all $n\in\N$ large enough since $L_n^{-1}=o(J_n^{-1})$. Finally, for the series of the high frequencies, in view of \eqref{Eq:GammaL3} and the boundedness of $t_0\in\ell^\infty_0$,  we have, arguing similarly to the previous case for the numerators,
\begin{align*}
	\sum_{l =L_n+1}^{\infty}&\gamma_l^{-2}l^{2s/d}
	\frac{\int_\R |t - t_{0,l}|^2e^{-\frac{n}{2\gamma_l^2}(t - t_{0,l})^2}e^{-|t|}dt}
	{\int_\R e^{-\frac{n}{2\gamma_l^2}(t - t_{0,l})^2}e^{-|t|} dt }\\
	&\lesssim 
	\sum_{l =L_n+1}^{\infty}\gamma_l^{-2}l^{2s/d}
	\frac{1}
	{\int_\R e^{-(t - t_{0,l})^2}e^{-|t|} dt }\\
	&\le
	\sum_{l =L_n+1}^{\infty}\gamma_l^{-2}l^{2s/d}
	 \frac{1}
	{e^{-|t_{0,l}|^2}\int_\R e^{-t^2}e^{-(2\sup_{l\ge1}|t_{0,l}| + 1)|t|} dt }\\
	&\lesssim
	\sum_{l =L_n+1}^{\infty}\gamma_l^{-2}l^{2s/d} 
	= \sum_{l =L_n+1}^{\infty}l^{-2(\beta-s)/d-1}\lesssim L_n^{-2(\beta-s)/d}
	\simeq n^{-2(\beta-s)/(2\beta+d)}.
\end{align*}

	Combining the three obtained bounds yields that for all sufficiently large $n\in\N$,
\begin{align*}
	\sum_{l =L_n+1}^{\infty}&\gamma_l^{-2}l^{2s/d}
	\frac{\int_\R |t - t_{0,l}|^2e^{-\frac{n}{2\gamma_l^2}(t - t_{0,l})^2}e^{-|t|}dt}
	{\int_\R e^{-\frac{n}{2\gamma_l^2}(t - t_{0,l})^2}e^{-|t|} dt }
	\lesssim n^{-2(\beta-s)/(2\beta+d)}
	M_n^{(3\beta-2s+d/2)/(2\beta+d)}
\end{align*}
concluding the proof in view of the arbitrariness of $M_n$.
\end{proof}

	In the next auxiliary lemma, we express the denominators appearing in \eqref{Eq:LaplPriorStart} in terms of the complementary error function, cf.~\cite[Chapter 7]{olver2010nist},
\begin{equation}
\label{Eq:ERFC}
	\erfc(z):=\frac{2}{\sqrt \pi}\int_z^\infty e^{-s^2}ds,
	\qquad z\in\R.
\end{equation}

\begin{lemma}\label{Lem:LaplPriorDenom}
For all $n\in\N$, $t_0\in\R$ and $\gamma>0$, it holds that, with $f_0:=t_0/\gamma$,
\bals
	\int_\R &e^{-\frac{n}{2\gamma^2}(t - t_0)^2}
	e^{-|t|} dt\\
	&=
	\frac{\sqrt{\pi}\gamma}{\sqrt {2n}} 
	e^{ \frac{\gamma^2}{2n}}
	\Bigg[e^{-t_0}
	\erfc\left(-\frac{\sqrt n}{\sqrt{2}}f_0 + \frac{\gamma}{\sqrt {2n}}\right)
	+
	e^{t_0}
	\erfc\left(\frac{\sqrt n}{\sqrt{2}}f_0 + \frac{\gamma}{\sqrt{2n}} 
	\right)
	\Bigg].
\end{align*}
\end{lemma}

\begin{proof}
Decompose the integral into
\begin{align*}
	\int_0^\infty e^{-\frac{n}{2\gamma ^2}(t - t_0 )^2}
	 e^{-t} dt
	&+\int_{-\infty}^0e^{-\frac{n}{2\gamma ^2}(-(-t) - t_0 )^2}
	 e^{-(-t)} dt\\
	 &=\int_0^\infty e^{-\frac{n}{2\gamma ^2}(t - t_0 )^2}
	 e^{-t} dt
	 +\int_0^\infty e^{-\frac{n}{2\gamma ^2}( \tau - ( - t_0) )^2}
	e^{- \tau} d \tau,
\end{align*}
having used the substitution $ \tau=-t$ in the second integral. The first equals, recalling the notation $ f_0 =t_0/\gamma$,
\begin{align*}
	&e^{-\frac{n}{2}f_0 ^2}\int_0^\infty e^{-\frac{n}{2\gamma ^2}t^2 
	+ (n f_0 /\gamma  - 1) t } dt\\
	&\ =e^{-\frac{n}{2}f_0 ^2+\frac{\gamma ^2}{2n}(n f_0 /\gamma  - 1)^2}
	 \int_0^\infty
	e^{-[(\frac{\sqrt n}{\sqrt{2}\gamma }t)^2
	-2\frac{\sqrt n}{\sqrt{2}\gamma }t(n f_0 /\gamma  - 1)\frac{\gamma }{\sqrt{2 n}} 
	+\frac{\gamma ^2}{2n}(n f_0 /\gamma  - 1)^2]}dt\\
	&\ =e^{-\gamma  f_0  + \frac{\gamma ^2}{2n}}
	\int_0^\infty
	e^{-[\frac{\sqrt n}{\sqrt{2}\gamma }t
	-\frac{1}{\sqrt{2n}}(n f_0  - \gamma )]^2}dt.
\end{align*} 
Using the substitution $s:=\frac{\sqrt n}{\sqrt{2}\gamma }t-\frac{1}{\sqrt{2n}}(n f_0  - \gamma )$, we then obtain
\begin{align*}
	\int_0^\infty e^{-\frac{n}{2\gamma ^2}(t - t_0 )^2}
	 e^{-t} dt
	 &=e^{-\gamma  f_0  + \frac{\gamma ^2}{2n}}
	\frac{\sqrt 2\gamma }{\sqrt n}\int_{-\frac{1}{\sqrt{2n}}(n f_0  - \gamma )}^\infty e^{-s^2}ds\\
	 &= e^{-t_0  + \frac{\gamma ^2}{2n}}
	\frac{\sqrt{\pi}\gamma }{\sqrt {2 n}}
	\erfc\left(-\frac{\sqrt{n}}{\sqrt{2}}f_0  
	+ \frac{\gamma }{\sqrt{2n}}\right).
\end{align*}
An application of the same formula with $t_0$ replaced by $-t_0$ then shows that
$$
	\int_0^\infty e^{-\frac{n}{2\gamma ^2}( \tau - ( - t_0) )^2}
	e^{- \tau} d \tau
	=
	e^{t_0  + \frac{\gamma ^2}{2n}}
	\frac{\sqrt{\pi}\gamma }{\sqrt {2 n}}
	\erfc\left(\frac{\sqrt{n}}{\sqrt{2}}f_0  
	+ \frac{\gamma }{\sqrt{2n}}\right)
$$
which combined with the previous display proves the claim.
\end{proof}

\paragraph{Acknowledgements.} 
Stefano Favaro and Matteo Giordano gratefully acknowledge the financial support from the Italian Ministry of Education, University and Research (MIUR), ``Dipartimenti di Eccellenza'' grant 2023-2027. Matteo Giordano was partially supported by MIUR, PRIN project 2022CLTYP4.

\bibliography{References}

\begin{thebibliography}{58}
\providecommand{\natexlab}[1]{#1}
\providecommand{\url}[1]{\texttt{#1}}
\expandafter\ifx\csname urlstyle\endcsname\relax
  \providecommand{\doi}[1]{doi: #1}\else
  \providecommand{\doi}{doi: \begingroup \urlstyle{rm}\Url}\fi

\bibitem[Abraham and Deo(2023)]{abraham2023deep}
K.~Abraham and N.~Deo.
\newblock Deep gaussian process priors for bayesian inference in nonlinear
  inverse problems.
\newblock \emph{arXiv preprint arXiv:2312.14294}, 2023.

\bibitem[Agapiou and Castillo(2024)]{agapiou2024heavy}
S.~Agapiou and I.~Castillo.
\newblock Heavy-tailed bayesian nonparametric adaptation.
\newblock \emph{The Annals of Statistics}, 52\penalty0 (4):\penalty0
  1433--1459, 2024.

\bibitem[Agapiou and Savva(2024)]{agapiou2024adaptive}
S.~Agapiou and A.~Savva.
\newblock Adaptive inference over besov spaces in the white noise model using
  p-exponential priors.
\newblock \emph{Bernoulli}, 30\penalty0 (3):\penalty0 2275--2300, 2024.

\bibitem[Agapiou and Wang(2024)]{agapiou2024laplace}
S.~Agapiou and S.~Wang.
\newblock Laplace priors and spatial inhomogeneity in bayesian inverse
  problems.
\newblock \emph{Bernoulli}, 30\penalty0 (2):\penalty0 878--910, 2024.

\bibitem[Agapiou et~al.(2018)Agapiou, Burger, Dashti, and Helin]{ABDH18}
S.~Agapiou, M.~Burger, M.~Dashti, and T.~Helin.
\newblock Sparsity-promoting and edge-preserving maximum {\it a posteriori}
  estimators in non-parametric {B}ayesian inverse problems.
\newblock \emph{Inverse Problems}, 34\penalty0 (4):\penalty0 045002, 37, 2018.
\newblock ISSN 0266-5611.
\newblock \doi{10.1088/1361-6420/aaacac}.
\newblock URL \url{https://doi-org.ezp.lib.cam.ac.uk/10.1088/1361-6420/aaacac}.

\bibitem[Agapiou et~al.(2021)Agapiou, Dashti, and Helin]{ADH21}
S.~Agapiou, M.~Dashti, and T.~Helin.
\newblock {Rates of contraction of posterior distributions based on
  p-exponential priors}.
\newblock \emph{Bernoulli}, 27\penalty0 (3):\penalty0 1616 -- 1642, 2021.
\newblock \doi{10.3150/20-BEJ1285}.
\newblock URL \url{https://doi.org/10.3150/20-BEJ1285}.

\bibitem[Ambrosio et~al.(2008)Ambrosio, Gigli, and
  Savar{\'e}]{ambrosio2008gradient}
L.~Ambrosio, N.~Gigli, and G.~Savar{\'e}.
\newblock \emph{Gradient flows: in metric spaces and in the space of
  probability measures}.
\newblock Springer Science \& Business Media, 2008.

\bibitem[Arbel et~al.(2013)Arbel, Gayraud, and Rousseau]{AGR13}
J.~Arbel, G.~Gayraud, and J.~Rousseau.
\newblock Bayesian optimal adaptive estimation using a sieve prior.
\newblock \emph{Scand. J. Stat.}, 40\penalty0 (3):\penalty0 549--570, 2013.
\newblock ISSN 0303-6898.
\newblock \doi{10.1002/sjos.12002}.
\newblock URL \url{https://doi.org/10.1002/sjos.12002}.

\bibitem[Bardet et~al.(2018)Bardet, Gozlan, Malrieu, and Zitt]{bardet18}
J.-B. Bardet, N.~Gozlan, F.~Malrieu, and P.-A. Zitt.
\newblock {Functional inequalities for Gaussian convolutions of compactly
  supported measures: Explicit bounds and dimension dependence}.
\newblock \emph{Bernoulli}, 24\penalty0 (1):\penalty0 333 -- 353, 2018.
\newblock \doi{10.3150/16-BEJ879}.
\newblock URL \url{https://doi.org/10.3150/16-BEJ879}.

\bibitem[Bioucas-Dias(2006)]{BD06}
J.~M. Bioucas-Dias.
\newblock Bayesian wavelet-based image deconvolution: a {GEM} algorithm
  exploiting a class of heavy-tailed priors.
\newblock \emph{IEEE Trans. Image Process.}, 15\penalty0 (4):\penalty0
  937--951, 2006.
\newblock ISSN 1057-7149.
\newblock \doi{10.1109/TIP.2005.863972}.
\newblock URL
  \url{https://ezproxy-prd.bodleian.ox.ac.uk:2102/10.1109/TIP.2005.863972}.

\bibitem[Breitung(2006)]{breitung2006asymptotic}
K.~W. Breitung.
\newblock \emph{Asymptotic approximations for probability integrals}.
\newblock Springer, 2006.

\bibitem[Brown et~al.(2002)Brown, Cai, Low, and Zhang]{brown2002asymptotic}
L.~D. Brown, T.~T. Cai, M.~G. Low, and C.-H. Zhang.
\newblock Asymptotic equivalence theory for nonparametric regression with
  random design.
\newblock \emph{The Annals of statistics}, 30\penalty0 (3):\penalty0 688--707,
  2002.

\bibitem[Bui-Thanh and Ghattas(2015)]{BG15}
T.~Bui-Thanh and O.~Ghattas.
\newblock A scalable algorithm for {MAP} estimators in {B}ayesian inverse
  problems with {B}esov priors.
\newblock \emph{Inverse Probl. Imaging}, 9\penalty0 (1):\penalty0 27--53, 2015.
\newblock ISSN 1930-8337.
\newblock \doi{10.3934/ipi.2015.9.27}.
\newblock URL \url{https://doi-org.ezp.lib.cam.ac.uk/10.3934/ipi.2015.9.27}.

\bibitem[Castillo and Nickl(2014)]{CN14}
I.~Castillo and R.~Nickl.
\newblock On the {B}ernstein--von {M}ises phenomenon for nonparametric {B}ayes
  procedures.
\newblock \emph{Ann. Statist.}, 42\penalty0 (5):\penalty0 1941--1969, 2014.

\bibitem[Cattiaux and Guillin(2022)]{cattiaux2022functional}
P.~Cattiaux and A.~Guillin.
\newblock Functional inequalities for perturbed measures with applications to
  log-concave measures and to some bayesian problems.
\newblock \emph{Bernoulli}, 28\penalty0 (4):\penalty0 2294--2321, 2022.

\bibitem[Chen et~al.(2018)Chen, Dunlop, Papaspiliopoulos, and
  Stuart]{chen2018dimension}
V.~Chen, M.~M. Dunlop, O.~Papaspiliopoulos, and A.~M. Stuart.
\newblock Dimension-robust mcmc in bayesian inverse problems.
\newblock \emph{arXiv preprint arXiv:1803.03344}, 2018.

\bibitem[Dashti and Stuart(2017)]{DS17}
M.~Dashti and A.~M. Stuart.
\newblock The {B}ayesian approach to inverse problems.
\newblock In \emph{Handbook of uncertainty quantification. {V}ol. 1, 2, 3},
  pages 311--428. Springer, Cham, 2017.

\bibitem[Dashti et~al.(2012)Dashti, Harris, and Stuart]{DHS12}
M.~Dashti, S.~Harris, and A.~M. Stuart.
\newblock Besov priors for {B}ayesian inverse problems.
\newblock \emph{Inverse Probl. Imaging}, 6\penalty0 (2):\penalty0 183--200,
  2012.
\newblock ISSN 1930-8337.
\newblock \doi{10.3934/ipi.2012.6.183}.
\newblock URL \url{http://dx.doi.org/10.3934/ipi.2012.6.183}.

\bibitem[Dolera and Mainini(2020)]{dolera2020uniform}
E.~Dolera and E.~Mainini.
\newblock On uniform continuity of posterior distributions.
\newblock \emph{Statistics \& Probability Letters}, 157:\penalty0 108627, 2020.

\bibitem[Dolera and Mainini(2023)]{dolera2023lipschitz}
E.~Dolera and E.~Mainini.
\newblock Lipschitz continuity of probability kernels in the optimal transport
  framework.
\newblock \emph{Ann. Inst. Henri Poincar{\'e} Probab. Stat}, 59\penalty0
  (4):\penalty0 1778--1812, 2023.

\bibitem[Dolera et~al.(2024)Dolera, Favaro, and Mainini]{dolera2024strong}
E.~Dolera, S.~Favaro, and E.~Mainini.
\newblock Strong posterior contraction rates via wasserstein dynamics.
\newblock \emph{Probability Theory and Related Fields}, 189:\penalty0 659--720,
  2024.

\bibitem[Donoho and Johnstone(1998)]{DJ98}
D.~L. Donoho and I.~M. Johnstone.
\newblock Minimax estimation via wavelet shrinkage.
\newblock \emph{Ann. Statist.}, 26\penalty0 (3):\penalty0 879--921, 1998.
\newblock ISSN 0090-5364.
\newblock \doi{10.1214/aos/1024691081}.
\newblock URL \url{https://doi.org/10.1214/aos/1024691081}.

\bibitem[Fischer and Steinwart(2020)]{fischer2020sobolev}
S.~Fischer and I.~Steinwart.
\newblock Sobolev norm learning rates for regularized least-squares algorithms.
\newblock \emph{Journal of Machine Learning Research}, 21\penalty0
  (205):\penalty0 1--38, 2020.

\bibitem[Ghosal and van~der Vaart(2007)]{GvdV07}
S.~Ghosal and A.~van~der Vaart.
\newblock Convergence rates of posterior distributions for non-i.i.d.
  observations.
\newblock \emph{Ann. Statist.}, 35\penalty0 (1):\penalty0 192--223, 2007.

\bibitem[Ghosal and van~der Vaart(2017)]{GvdV17}
S.~Ghosal and A.~W. van~der Vaart.
\newblock \emph{Fundamentals of Nonparametric Bayesian Inference}.
\newblock Cambridge University Press, New York, 2017.

\bibitem[Ghosal et~al.(2000)Ghosal, Ghosh, and van~der Vaart]{GGvdV00}
S.~Ghosal, J.~K. Ghosh, and A.~W. van~der Vaart.
\newblock Convergence rates of posterior distributions.
\newblock \emph{Ann. Statist.}, 28\penalty0 (2):\penalty0 500--531, 2000.

\bibitem[Gin{\'e} and Nickl(2011)]{GN11}
E.~Gin{\'e} and R.~Nickl.
\newblock Rates of contraction for posterior distributions in {$L^r$}-metrics,
  {$1\leq r\leq\infty$}.
\newblock \emph{Ann. Statist.}, 39\penalty0 (6):\penalty0 2883--2911, 2011.

\bibitem[Gin\'e and Nickl(2016)]{GN16}
E.~Gin\'e and R.~Nickl.
\newblock \emph{Mathematical foundations of infinite-dimensional statistical
  models}.
\newblock Cambridge University Press, New York, 2016.
\newblock ISBN 978-1-107-04316-9.
\newblock \doi{10.1017/CBO9781107337862}.
\newblock URL \url{http://dx.doi.org/10.1017/CBO9781107337862}.

\bibitem[Giordano(2023)]{giordano23besov}
M.~Giordano.
\newblock {Besov-Laplace priors in density estimation: optimal posterior
  contraction rates and adaptation}.
\newblock \emph{Electronic Journal of Statistics}, 17\penalty0 (2):\penalty0
  2210 -- 2249, 2023.
\newblock \doi{10.1214/23-EJS2161}.
\newblock URL \url{https://doi.org/10.1214/23-EJS2161}.

\bibitem[Giordano and Ray(2022)]{GR22}
M.~Giordano and K.~Ray.
\newblock Nonparametric {B}ayesian inference for reversible multidimensional
  diffusions.
\newblock \emph{Ann. Statist.}, 50\penalty0 (5):\penalty0 2872--2898, 2022.
\newblock ISSN 0090-5364,2168-8966.
\newblock \doi{10.1214/22-aos2213}.
\newblock URL \url{https://doi.org/10.1214/22-aos2213}.

\bibitem[Giordano et~al.(2022)Giordano, Ray, and
  Schmidt-Hieber]{giordano2022inability}
M.~Giordano, K.~Ray, and J.~Schmidt-Hieber.
\newblock On the inability of gaussian process regression to optimally learn
  compositional functions.
\newblock \emph{Advances in Neural Information Processing Systems},
  35:\penalty0 22341--22353, 2022.

\bibitem[H\"{a}rdle et~al.(1998)H\"{a}rdle, Kerkyacharian, Picard, and
  Tsybakov]{HKPT98}
W.~H\"{a}rdle, G.~Kerkyacharian, D.~Picard, and A.~Tsybakov.
\newblock \emph{Wavelets, approximation, and statistical applications}, volume
  129 of \emph{Lecture Notes in Statistics}.
\newblock Springer-Verlag, New York, 1998.
\newblock ISBN 0-387-98453-4.
\newblock \doi{10.1007/978-1-4612-2222-4}.
\newblock URL \url{https://doi.org/10.1007/978-1-4612-2222-4}.

\bibitem[Holsclaw et~al.(2013)Holsclaw, Sans{\'o}, Lee, Heitmann, Habib,
  Higdon, and Alam]{holsclaw2013gaussian}
T.~Holsclaw, B.~Sans{\'o}, H.~K. Lee, K.~Heitmann, S.~Habib, D.~Higdon, and
  U.~Alam.
\newblock Gaussian process modeling of derivative curves.
\newblock \emph{Technometrics}, 55\penalty0 (1):\penalty0 57--67, 2013.

\bibitem[Jia et~al.(2016)Jia, Peng, and Gao]{JPG16}
J.~Jia, J.~Peng, and J.~Gao.
\newblock Bayesian approach to inverse problems for functions with a
  variable-index {B}esov prior.
\newblock \emph{Inverse Problems}, 32\penalty0 (8):\penalty0 085006, 32, 2016.
\newblock ISSN 0266-5611.
\newblock \doi{10.1088/0266-5611/32/8/085006}.
\newblock URL
  \url{https://doi-org.ezp.lib.cam.ac.uk/10.1088/0266-5611/32/8/085006}.

\bibitem[Kekkonen et~al.(2023)Kekkonen, Lassas, Saksman, and Siltanen]{KLSS23}
H.~Kekkonen, M.~Lassas, E.~Saksman, and S.~Siltanen.
\newblock Random tree {B}esov priors---towards fractal imaging.
\newblock \emph{Inverse Probl. Imaging}, 17\penalty0 (2):\penalty0 507--531,
  2023.
\newblock ISSN 1930-8337.

\bibitem[Kolehmainen et~al.(2012)Kolehmainen, Lassas, Niinim\"{a}ki, and
  Siltanen]{KLNS12}
V.~Kolehmainen, M.~Lassas, K.~Niinim\"{a}ki, and S.~Siltanen.
\newblock Sparsity-promoting {B}ayesian inversion.
\newblock \emph{Inverse Problems}, 28\penalty0 (2):\penalty0 025005, 28, 2012.
\newblock ISSN 0266-5611.
\newblock \doi{10.1088/0266-5611/28/2/025005}.
\newblock URL \url{https://doi.org/10.1088/0266-5611/28/2/025005}.

\bibitem[Lassas and Siltanen(2004)]{LS04}
M.~Lassas and S.~Siltanen.
\newblock Can one use total variation prior for edge-preserving {B}ayesian
  inversion?
\newblock \emph{Inverse Problems}, 20\penalty0 (5):\penalty0 1537--1563, 2004.
\newblock ISSN 0266-5611.
\newblock \doi{10.1088/0266-5611/20/5/013}.
\newblock URL
  \url{https://doi-org.ezp.lib.cam.ac.uk/10.1088/0266-5611/20/5/013}.

\bibitem[Lassas et~al.(2009)Lassas, Saksman, and Siltanen]{LSS09}
M.~Lassas, E.~Saksman, and S.~Siltanen.
\newblock Discretization-invariant {B}ayesian inversion and {B}esov space
  priors.
\newblock \emph{Inverse Probl. Imaging}, 3\penalty0 (1):\penalty0 87--122,
  2009.
\newblock ISSN 1930-8337.
\newblock \doi{10.3934/ipi.2009.3.87}.
\newblock URL \url{http://dx.doi.org/10.3934/ipi.2009.3.87}.

\bibitem[Lember and van~der Vaart(2007)]{lember2007universal}
J.~Lember and A.~van~der Vaart.
\newblock On universal bayesian adaptation.
\newblock \emph{Statistics \& Decisions. International Mathematical Journal for
  Stochastic Methods and Models}, 25\penalty0 (2):\penalty0 127--152, 2007.

\bibitem[Leporini and Pesquet(2001{\natexlab{a}})]{LP01}
D.~Leporini and J.-C. Pesquet.
\newblock Bayesian wavelet denoising: Besov priors and non-{G}aussian noises.
\newblock \emph{Signal Processing}, 81\penalty0 (1):\penalty0 55--67,
  2001{\natexlab{a}}.
\newblock ISSN 0165-1684.
\newblock \doi{https://doi.org/10.1016/S0165-1684(00)00190-0}.
\newblock URL
  \url{https://www.sciencedirect.com/science/article/pii/S0165168400001900}.
\newblock Special section on Markov Chain Monte Carlo (MCMC) Methods for Signal
  Processing.

\bibitem[Leporini and Pesquet(2001{\natexlab{b}})]{leporini2001bayesian}
D.~Leporini and J.-C. Pesquet.
\newblock Bayesian wavelet denoising: Besov priors and non-gaussian noises.
\newblock \emph{Signal processing}, 81\penalty0 (1):\penalty0 55--67,
  2001{\natexlab{b}}.

\bibitem[Liu and Li(2022)]{liu2022optimal}
Z.~Liu and M.~Li.
\newblock Optimal plug-in gaussian processes for modelling derivatives.
\newblock \emph{arXiv preprint arXiv:2210.11626}, 2022.

\bibitem[Niinim{\"a}ki et~al.(2007)Niinim{\"a}ki, Siltanen, and
  Kolehmainen]{niinimaki2007bayesian}
K.~Niinim{\"a}ki, S.~Siltanen, and V.~Kolehmainen.
\newblock Bayesian multiresolution method for local tomography in dental x-ray
  imaging.
\newblock \emph{Physics in Medicine \& Biology}, 52\penalty0 (22):\penalty0
  6663, 2007.

\bibitem[Olver et~al.(2010)Olver, Lozier, Boisvert, and Clark]{olver2010nist}
F.~W. Olver, D.~W. Lozier, R.~F. Boisvert, and C.~W. Clark.
\newblock \emph{NIST handbook of mathematical functions}.
\newblock Cambridge university press, 2010.

\bibitem[Panaretos and Zemel(2019)]{panaretos2019statistical}
V.~M. Panaretos and Y.~Zemel.
\newblock Statistical aspects of wasserstein distances.
\newblock \emph{Annual review of statistics and its application}, 6\penalty0
  (1):\penalty0 405--431, 2019.

\bibitem[Rantala et~al.(2006)Rantala, V{\"a}nsk{\"a}, J{\"a}rvenp{\"a}{\"a},
  Kalke, Lassas, Moberg, and Siltanen]{RVJKLMS06}
M.~Rantala, S.~V{\"a}nsk{\"a}, S.~J{\"a}rvenp{\"a}{\"a}, M.~Kalke, M.~Lassas,
  J.~Moberg, and S.~Siltanen.
\newblock Wavelet-based reconstruction for limited-angle x-ray tomography.
\newblock \emph{IEEE Transactions on Medical Imaging}, 25\penalty0
  (2):\penalty0 210--217, 2006.
\newblock ISSN 0278-0062.
\newblock \doi{10.1109/TMI.2005.862206}.

\bibitem[Ray(2013)]{R13}
K.~Ray.
\newblock Bayesian inverse problems with non-conjugate priors.
\newblock \emph{Electron. J. Stat.}, 7:\penalty0 2516--2549, 2013.

\bibitem[Rei\ss(2008)]{R08}
M.~Rei\ss.
\newblock Asymptotic equivalence for nonparametric regression with multivariate
  and random design.
\newblock \emph{Ann. Statist.}, 36\penalty0 (4):\penalty0 1957--1982, 2008.
\newblock ISSN 0090-5364.
\newblock \doi{10.1214/07-AOS525}.
\newblock URL \url{http://dx.doi.org/10.1214/07-AOS525}.

\bibitem[Rudin et~al.(1992)Rudin, Osher, and Fatemi]{ROF92}
L.~I. Rudin, S.~Osher, and E.~Fatemi.
\newblock Nonlinear total variation based noise removal algorithms.
\newblock volume~60, pages 259--268. 1992.
\newblock \doi{10.1016/0167-2789(92)90242-F}.
\newblock URL
  \url{https://doi-org.ezp.lib.cam.ac.uk/10.1016/0167-2789(92)90242-F}.
\newblock Experimental mathematics: computational issues in nonlinear science
  (Los Alamos, NM, 1991).

\bibitem[Sakhaee and Entezari(2015{\natexlab{a}})]{SE15}
E.~Sakhaee and A.~Entezari.
\newblock {Spline-based sparse tomographic reconstruction with Besov priors}.
\newblock In S.~Ourselin and M.~A. Styner, editors, \emph{Medical Imaging 2015:
  Image Processing}, volume 9413, pages 101 -- 108. International Society for
  Optics and Photonics, SPIE, 2015{\natexlab{a}}.
\newblock \doi{10.1117/12.2082797}.
\newblock URL \url{https://doi.org/10.1117/12.2082797}.

\bibitem[Sakhaee and Entezari(2015{\natexlab{b}})]{sakhaee2015spline}
E.~Sakhaee and A.~Entezari.
\newblock Spline-based sparse tomographic reconstruction with besov priors.
\newblock In \emph{Medical Imaging 2015: Image Processing}, volume 9413, pages
  101--108. SPIE, 2015{\natexlab{b}}.

\bibitem[Shen and Wasserman(2001)]{shen2001rates}
X.~Shen and L.~Wasserman.
\newblock Rates of convergence of posterior distributions.
\newblock \emph{Annals of Statistics}, pages 687--714, 2001.

\bibitem[Stein(2012)]{stein2012interpolation}
M.~L. Stein.
\newblock \emph{Interpolation of spatial data: some theory for kriging}.
\newblock Springer Science \& Business Media, 2012.

\bibitem[Stuart(2010)]{S10}
A.~M. Stuart.
\newblock Inverse problems: a {B}ayesian perspective.
\newblock \emph{Acta Numer.}, 19:\penalty0 451--559, 2010.

\bibitem[Tsybakov(2009)]{T09}
A.~B. Tsybakov.
\newblock \emph{Introduction to nonparametric estimation}.
\newblock Springer Series in Statistics. Springer, New York, 2009.
\newblock ISBN 978-0-387-79051-0.
\newblock \doi{10.1007/b13794}.
\newblock URL \url{https://doi-org.ezp.lib.cam.ac.uk/10.1007/b13794}.
\newblock Revised and extended from the 2004 French original, Translated by
  Vladimir Zaiats.

\bibitem[van~der Vaart and van Zanten(2008)]{vdVvZ08}
A.~W. van~der Vaart and J.~H. van Zanten.
\newblock Rates of contraction of posterior distributions based on {G}aussian
  process priors.
\newblock \emph{Ann. Statist.}, 36\penalty0 (3):\penalty0 1435--1463, 2008.

\bibitem[van~der Vaart and van Zanten(2009)]{vdVvZ09}
A.~W. van~der Vaart and J.~H. van Zanten.
\newblock Adaptive {B}ayesian estimation using a {G}aussian random field with
  inverse gamma bandwidth.
\newblock \emph{Ann. Statist.}, 37\penalty0 (5B):\penalty0 2655--2675, 2009.
\newblock ISSN 0090-5364.
\newblock \doi{10.1214/08-AOS678}.
\newblock URL \url{https://doi.org/10.1214/08-AOS678}.

\bibitem[V\"{a}nsk\"{a} et~al.(2009)V\"{a}nsk\"{a}, Lassas, and
  Siltanen]{VLS09}
S.~V\"{a}nsk\"{a}, M.~Lassas, and S.~Siltanen.
\newblock Statistical {X}-ray tomography using empirical {B}esov priors.
\newblock \emph{Int. J. Tomogr. Stat.}, 11\penalty0 (S09):\penalty0 3--32,
  2009.
\newblock ISSN 0972-9976.

\end{thebibliography}

\end{document}